\documentclass[12pt]{elsarticle}
\usepackage{graphicx}        % standard LaTeX graphics tool
\usepackage{graphics}
\usepackage{amssymb, latexsym}
\usepackage{amsmath, amsfonts}
\usepackage{amsthm}
\usepackage{array}
\usepackage{color, xcolor}
\usepackage{subfigure}

\newtheorem{theorem}{Theorem}[section]
\newtheorem{lemma}[theorem]{Lemma}
\newtheorem{proposition}[theorem]{Proposition}

\newtheorem{remark}[theorem]{Remark}

\def \Real{\mathbb{R}}
\def \bfA{{\mathbf A}}
\def \bfB{{\mathbf B}}
\def \bfC{{\mathbf C}}
\def \bfD{{\mathbf D}}
\def \bfE{{\mathbf E}}
\def \bfF{{\mathbf F}}

\def \bfH{{\mathbf H}}
\def \bfI{{\mathbf I}}

\def \bfK{{\mathbf K}}
\def \bfL{{\mathbf L}}
\def \bfM{{\mathbf M}}
\def \bfN{{\mathbf N}}

\def \bfV{{\mathbf V}}

\def \bfX{{\mathbf X}}

\def \bfv{{\mathbf v}}
\def \bfzero{{\mathbf 0}}
\def \bfa{{\mathbf a}}
\def \bfb{{\mathbf b}}
\def \bfc{{\mathbf c}}
\def \bfd{{\mathbf d}}

\def \bff{{\mathbf f}}

\def \bfn{{\mathbf n}}

\def \bft{{\mathbf t}}
\def \bfu{{\mathbf u}}
\def \bfv{{\mathbf v}}
\def \bfw{{\mathbf w}}
\def \bfx{{\mathbf x}}

\def \bfgamma{{\mathbf{\gamma}}}
\def \bfudot{\dot{\bfu}}
\def \bfetaddot{\ddot{\bfeta}}
\def \pdot{\dot{p}}
\def \bfetadot{\dot{\bfeta}}

\def \bfalpha{\mbox{\boldmath $\alpha$}}

\def \bfbeta{\mbox{\boldmath $\beta$}}
\def \bfsigma{\mbox{\boldmath $\sigma$}}
\def \bfSigma{\mbox{\boldmath $\Sigma$}}

\def \bfgamma{\mbox{\boldmath $\gamma$}}

\def \bftheta{\mbox{\boldmath $\theta$}}
\def \bfeta{\mbox{\boldmath $\eta$}}
\def \bfxi{\mbox{\boldmath $\xi$}}
\def \bfphi{\mbox{\boldmath $\phi$}}

\def \calC{\mathcal{C}}

\def \calD{\mathcal{D}}
\def \bfcalD{\mathbf{\calD}}

\def \calF{\mathcal{F}}

\def\calN{\mathcal{N}}
\def \rmspan{\mathrm{span}}

\def \Lbb{\mathbb{L}}
\def \Hbb{\mathbb{H}}

\def \dive{{\rm div}}
\usepackage{amssymb}
\begin{document}

\begin{frontmatter}



\title{Analysis of the coupled Navier-Stokes/Biot problem}

\author{Aycil Cesmelioglu}
\address{Oakland University, Department of Mathematics and Statistics, 146 Library Drive, Rochester, MI 48309}

\ead{cesmelio@oakland.edu.}

%\tableofcontents
\begin{abstract}
We analyze a weak formulation of the coupled  problem defining the interaction between a free fluid and a poroelastic structure. The problem is fully dynamic and is governed by the time-dependent incompressible Navier-Stokes equations and the Biot equations. Under a small data assumption, existence and uniqueness results are proved and a priori estimates are provided. %\today
\end{abstract}


\end{frontmatter}


\section{Introduction.}
We consider a fully dynamic model for the interaction of an incompressible Newtonian fluid with a poroelastic material where the boundary is assumed to be fixed. The fluid flow is governed by the time-dependent incompressible Navier-Stokes equations. For the poroelastic material we use the Biot system with appropriate flow and stress couplings on the interface between the fluid and the poroelastic regions. This problem is a fully dynamic coupled system of mixed hyperbolic-parabolic type and inherits all the difficulties mathematically and numerically involved in the standard fluid-structure interaction and Stokes/Navier-Stokes-Darcy coupling. 

The literature is rich in works on related coupled problems and here we provide only a partial list of relevant publications. One related problem deals with the interaction of an incompressible fluid with a porous material and modeled by the coupling of the Stokes/Navier-Stokes equations to the Darcy equations. The steady-state case of this problem is analyzed mathematically in \cite{GR2009, DQ2009, BDQ2010} and the time-dependent case in \cite{CR2008, CGHW2010, CGR2013}. 

Another related problem is that of the fluid-structure interaction. 
The analysis of a weak solution for the time dependent coupling of the Stokes and the linear elasticity equations is discussed in \cite{DGHL2003} and for the flow of a coupling of time-dependent 2D incompressible NSE to the linearly viscoelastic or the linearly elastic Koiter shell in \cite{MuhaCanic2013}. The two layered structure version of this problem is discussed in \cite{MuhaCanic2014, MuhaCanic2016}. 

In geosciences, aquifers and oil/gas reservoirs are porous and deformable affecting groundwater and oil/gas flow, respectively \cite{MJR2005, LDQ2011, GMSWW2014}. In biomedical sciences, blood flow is influenced by the porous and deformable nature of the arterial wall \cite{BQQ2009, BGN2014, BYZ2015}. Therefore, mathematical models that are used to simulate these flow problems must account for both the effects of porosity and elasticity. The Navier-Stokes/Biot system is investigated numerically in \cite{BQQ2009} using a monolithic and a domain decomposition technique and the Stokes/Biot system is investigated in \cite{BYZZ2015} using an operator splitting approach and in \cite{CLQWY2016} using an optimization based decoupling strategy. In \cite{Showalter2005}, variational formulations for the the Stokes/Biot system are developed using semi-group methods. Also a two-layered version was studied in \cite{BYZ2015}.

In this paper, we focus on the coupling of the fully dynamic incompressible Navier-Stokes equations (for the free fluid) with the Biot system (for the poroelastic structure completely saturated with fluid). This coupled problem is the nonlinear version of the problem presented in \cite{BYZZ2015, CLQWY2016}. We construct a weak formulation and show the existence and uniqueness (local) of its solution under small data assumption. We note that this small data assumption is not needed if the fluid is represented by the linear Stokes equations rather than the Navier-Stokes equations and the result would also be global. We assume that the boundaries and the interface between the fluid and the poroelastic material are fixed. The proof proceeds by constructing a semi-discrete Galerkin approximations, obtaining the necessary a priori estimates, and passing to the limit. To the author's knowledge, there is no such analysis for this fully dynamic nonlinear coupled system.

The outline of this paper is as follows: In Section 2, we introduce the equations governing the problem and appropriate interface, boundary and initial conditions. The next section is devoted to notation and some well-known results that are used in the forthcoming sections. Section 4 sets the assumptions on data, presents the weak formulation and shows that it is equivalent to the problem. Section 5 summarizes the main result of the paper. Section 6 contains the proof of the existence and uniqueness results and a priori estimates for the weak solution.

\section{Fluid-poroelastic model equations}\label{sec: model}
\setcounter{equation}{0}
Let $\Omega\subset \mathbb{R}^d$, $d=2,3$ be an open bounded domain with Lipschitz continuous boundary $\partial \Omega$. The domain $\Omega$ is made up of two regions $\Omega_f$, the fluid region, and
$\Omega_p$, the poroelastic region, separated by a common interface
$\Gamma_{I} = \partial \Omega_f \cap \partial \Omega_p$. Both $\Omega_1$ and $\Omega_2$ are assumed to be Lipschitz. See Figure ~\ref{fig:domain}.
The first region $\Omega_f$ is occupied by a free fluid and has boundary $\Gamma_f$ such that
$\Gamma_f=\Gamma^{in}_f \cup \Gamma^{out}_f \cup \Gamma^{ext}_f
\cup \Gamma_{I}$, where $\Gamma^{in}_f$  and $\Gamma^{out}_f$ represent the inlet and outlet boundary, respectively.
The second region $\Omega_p$ is occupied by a saturated poroelastic structure with boundary $\Gamma_p$ such that $\Gamma_p=\Gamma_p^{s}\cup \Gamma^{ext}_p
\cup \Gamma_{I}$, where $\Gamma^{s}_p\cup \Gamma^{ext}_p$ represents the outer structure boundary.

\begin{figure}[ht!]
\begin{center}
\includegraphics[scale=.25]{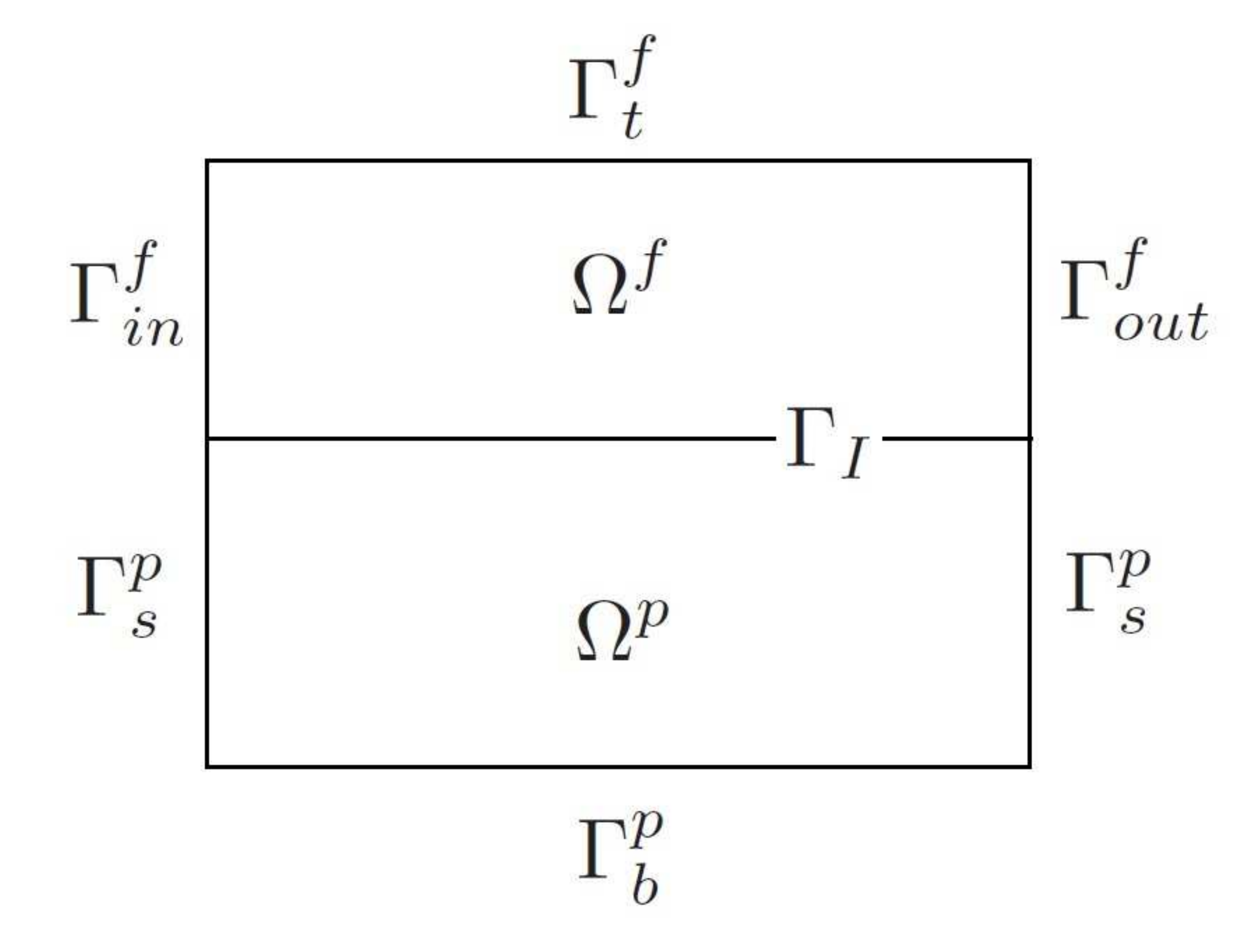}
\caption{Fluid-poroelastic domain}
\label{fig:domain}
\end{center}
\end{figure}

Fluid flow is governed by the time-dependent incompressible Navier-Stokes equations:
\begin{subequations}\label{NSE}
\begin{align}
\rho_f  \bfudot_f
 -2  \mu_f\, \nabla \cdot \bfD(\bfu_f) + \rho_f\bfu_f\cdot\nabla\bfu_f+ \nabla
p_f &= \bff_f  \quad \mbox{in } \Omega_f \times (0,T)\,, \label{moment}
\\[1.5ex]
\nabla \cdot \; \bfu_f & = 0  \quad \mbox{in } \Omega_f\times (0,T) \,. \label{mass}
\end{align}
\end{subequations}
Here $\bfu_f$ denotes the
velocity vector of the fluid, $p_f$ the pressure of the fluid, $\rho_f$ the density
of the fluid, $\mu_f$ the constant fluid viscosity, and $\bff_f$ the body force acting on the fluid. We used the dot above a symbol to denote the time derivative.
The strain rate tensor $\bfD(\bfu_f)$ is defined by:
\begin{equation*}
\bfD(\bfu_f) = \frac{1}{2}\left( \nabla \bfu_f + (\nabla \bfu_f)^T \right ).
\end{equation*}
The Cauchy stress tensor is given by:
\begin{align}
\bfsigma_f =  2\mu_{f} D(\bfu_f)  - p_f \bfI. 
\end{align}
So \eqref{moment} can also be written as
\[
\rho_f  \bfudot_f
 -\nabla \cdot \bfsigma_f + \rho_f\bfu_f\cdot\nabla\bfu_f = \bff_f.
\]
Equation (\ref{moment}) represents the conservation of linear momentum, while equation (\ref{mass}) is the incompressibility condition that represents the conservation of mass.

The poroelastic system is a fully dynamic coupled system of mixed hyperbolic-parabolic
type represented by the Biot model \cite{Biot1941, Biot1955}:
\begin{subequations}\label{biot}
\begin{align}
\rho_p  \bfetaddot
 -2 \mu_s \, \nabla \cdot \bfD(\bfeta)  - \lambda_s \nabla (\nabla \cdot \bfeta)  + \alpha \nabla p_p &= \bff_s \quad \mbox{in } \Omega_p\times (0,T) \,, \label{structure}
 \\[1.5ex]
 % \bfu_p + \bfK\nabla p_p &= 0 \quad \mbox{in } \Omega_p\times (0,T) \,, \label{darcy}
% \\[1.5ex]
(s_0 \pdot_p+ \alpha  \nabla \cdot \bfetadot ) -\nabla \cdot \bfK\nabla p_p% + \nabla \cdot  \bfu_p 
&=  f_p \quad \mbox{in } \Omega_p\times (0,T) \,, \label{mass1}
 \end{align}
 \end{subequations}
where $\bfeta$ is the displacement of the structure, $p_p$ is the pore pressure of the fluid, and $\bfu_p$ is the fluid velocity in the pores. Here, $f_p$ is the source/sink term and $\bff_s$ is the body force. The parameters $\nu_s$ and $\lambda_s$ denote the Lam\'e constants for the solid skeleton.
The density of the saturated medium is denoted by $\rho_p$, and the hydraulic conductivity by $\bfK$. In the Biot model, the first equation, \eqref{structure}, is the momentum equation for the balance of forces and the second equation, \eqref{mass1}, is the diffusion equation of fluid mass.
The total stress tensor for the poroelastic structure is given by:
\begin{align*}
\bfsigma_p = \bfsigma_p^E - \alpha \, p_p \bfI,
\end{align*}
where $\bfsigma_p^E$ is the elasticity stress tensor defined by $\bfsigma_p^E = 2 \mu_s \bfD(\bfeta) + \lambda_s (\nabla \cdot \bfeta )\bfI$.
Therefore, \eqref{structure} can also be written as
\[
\rho_p  \bfetaddot
-\nabla\cdot \bfsigma_p= \bff_s.
\]
The constrained specific storage coefficient is denoted by $s_0$ and the Biot-Willis constant by $\alpha$, the latter is usually close to unity. In the subsequent discussion, we assume that the motion of the structure is small enough so that the domain is fixed at its reference position. All the physical parameters are assumed to be constant in space and time. Next, we prescribe boundary, interface and initial conditions where $\bfn_f$ and $\bfn_p$ denote the outward unit normal vectors of $\Omega_f$ and $\Omega_p$, respectively and $\bfn_{\Gamma}$ is the unit normal vector of the interface $\Gamma_I$ pointing from $\Omega_f$ to $\Omega_p$. Hence, $\bfn_{\Gamma_I}=\bfn_f|_{\Gamma_I}=-\bfn_p|_{\Gamma_I}$. Furthermore $\bft_{\Gamma}^l$, $l=1,\hdots,d-1$ denotes an orthonormal set of unit vectors on the tangent plane to $\Gamma_I$.
\subsubsection*{Boundary conditions:}
Since the boundary conditions have no significant effect on the fluid poroelastic interaction, for simplicity they are chosen such that the normal fluid stress is prescribed on the inlet and outlet boundaries, the poroelastic structure is assumed to be fixed at the inlet and outlet boundaries and have zero tangential displacement on the external structure boundary, that is, 
\begin{subequations}\label{boundry}
\begin{alignat}{3}
&\mbox{ on } \, \Gamma_f^{in}\times (0,T) &,&\hspace{2cm} \bfsigma_f \bfn_f   = - P_{in}(t)\bfn_f , \label{bc_f_in} \\[1.5ex]
 &\mbox{ on }  \, \Gamma_f^{out}\times (0,T) &,&\hspace{2cm} \bfsigma_f  \bfn_f =  \,\bfzero, \label{bc_f_out} \\[1.5ex]
&\mbox{ on }  \, \Gamma_f^{ext}\times (0,T) &,&\hspace{2cm} \bfu_f  =   \,\bfzero
 \,, \label{bc_f_ext} \\[1.5ex]
%\bfu_p \cdot \bfn_p  = 0
&\mbox{ on }  \,\Gamma_p^{s}\times (0,T)&,&\hspace{2cm} \bfK\nabla p_p\cdot \bfn_p=0, \; \bfeta =  \,\bfzero, \label{bc_p_s} \,  \\[1.5ex]
%\bfu_p \cdot \bfn_p = 0
&\mbox{ on } \, \Gamma_p^{ext}\times (0,T) &,&\hspace{2cm}
\left\{ 
\begin{array}{c}p_p=0, \; \bfn_p\cdot \bfsigma_p^E  \bfn_p = \, \bfzero,
 \label{bc_p_ext} \,\\
 \;\bft_p^l\cdot \bfeta=0, \,1\leq l\leq d-1.
 \end{array}
 \right.\quad && 
\end{alignat}
\end{subequations}
\subsubsection*{Initial conditions:} 
As initial conditions, we assume that everything is at rest in the beginning.
\begin{equation}\label{initial}
\text{At}~ t=0: \quad \bfu_f =  \bfzero,  \quad  {p_p} = 0, \quad  \bfeta=\bfzero, \quad \dot{\bfeta}=\bfzero \,.
\end{equation}
\subsubsection*{Interface conditions on $\Gamma_I\times (0,T)$:}
The interface conditions are given by 
\begin{subequations}\label{interface}
\begin{align}
\bfu_f \cdot \bfn_{\Gamma} & =  (\bfetadot -\bfK\nabla p_p) \cdot \bfn_{\Gamma}\,, \label{inter1} \\[1.5ex]
\bfsigma_f  \bfn_{\Gamma} & =   \bfsigma_p  \bfn_{\Gamma}
 \,, \label{inter2} \\[1.5ex]
 \bfn_{\Gamma} \cdot \bfsigma_f  \bfn_{\Gamma} & =  - p_p
 \,, \label{inter3} \\[1.5ex]
\bfn_{\Gamma} \cdot \bfsigma_f  \bft_{\Gamma}^l & =  - \beta (\bfu_f - \bfetadot)
 \cdot \bft_{\Gamma}^l  \, ,\,1\leq l\leq d-1,  \label{inter4}
\end{align}
\end{subequations}
where $\beta$ denotes the resistance parameter in the tangential direction.
Condition \eqref{inter1} is the continuity of normal flux that satisfies mass conservation, condition \eqref{inter2} is the balance of stresses, that is, the total stresses of the fluid and the poroelastic medium must match at the interface. Condition \eqref{inter3} guarantees the balance of normal components of the stress in the fluid phase across the interface. Finally condition \eqref{inter4} is the Beavers-Joseph-Saffman condition \cite{BJ, Saffman, Jager} which assumes that the tangential stress of the fluid is proportional to the slip rate. More details on the interface conditions can be found in \cite{MW2012, Showalter2005}. 
\section{Notation and useful results.}
Let $\Omega\subset \Real^d$, $d=2, 3$ be a bounded, open connected domain with a Lipschitz continuous boundary $\partial \Omega$. Let $\calD(\Omega)$ be the space of all infinitely differentiable functions with compact support in $\Omega$. For $s\in \mathbb{R}$, $H^s(\Omega)$ denotes the standard Sobolev space of order $s$ equipped with its standard seminorm $|\cdot|_{s,\Omega}$ and norm $\|\cdot\|_{s,\Omega}$.
We denote the vector- and matrix-valued Sobolev spaces as follows:
\[
\bfH^s(\Omega):=[H^s(\Omega)]^d,\quad \Hbb^s(\Omega):=[H^s(\Omega)]^{d\times d}
\]
and still write $|\cdot|_{s,\Omega}$ and $\|\cdot\|_{s,\Omega}$ for the corresponding seminorms and norms.
When $s=0$, instead of $H^0(\Omega), \bfH^0(\Omega)$ and $\Hbb^0(\Omega)$ we write $L^2(\Omega), \bfL^2(\Omega)$ and $\Lbb^2(\Omega)$ and instead of $|\cdot|_{0,\Omega}$ and $\|\cdot\|_{0,\Omega}$ we write $|\cdot|_{\Omega}$ and $\|\cdot\|_{\Omega}$. The scalar product of $L^2(\Omega)$ is denoted by $(\cdot, \cdot)_{\Omega}$.
The spaces $\bfH(\dive;\Omega)=\{\bfv\in \bfL^2(\Omega): \nabla \cdot\bfv\in L^2(\Omega)\}$ and $\bfH^{3/2}({\dive};\Omega)=\{\bfv\in \bfL^2(\Omega): \nabla \cdot\bfv\in L^{3/2}(\Omega)\}$ are equipped with the graph norm.
If $\Gamma\subset \partial \Omega$, and $v\in H^{1/2}(\Gamma)$, we define
the extension $\tilde{v}$ of $v$ as $\tilde{v}=v$ on $\Gamma$, $\tilde{v}=0$ on $\partial \Omega\backslash\Gamma$ and define the space of traces of all functions of $H^1(\Omega)$ that vanish on $\partial \Omega\backslash \Gamma$ as follows:
\[ 
H_{00}^{1/2}(\Gamma)=\{v\in L^2(\Gamma): \tilde{v}\in H^{1/2}(\partial \Omega)\}.
\]   
We also define for any $1\leq r\leq \infty$,
\[
L^r(a, b;X)=\{f \mbox{ measurable in }: \|f\|_{L^r(a,b;X)}<\infty\}
\]
equipped with the norm
\[
\|f\|_{L^r(a,b;X)}=\left(\int_a^b\|f(t)\|^r_X dt\right)^{1/r}
\]
for $1\leq r<\infty$ and 
\[
\|f\|_{L^{\infty}(a,b;X)}={\rm ess} \sup_{t\in [a,b]}\|f(t)\|_X.
\]
Furthermore, 
$
\calC(0,T;X)
$ denotes the set of all functions that are continuous into $X$ and finally we define 
\[
H^1(a,b;X)=\{f\in L^2(a,b;X): \dot{f}\in L^2(a,b;X)\}.
\]
\subsubsection{Useful results.} \label{usefulresults}
Here we state inequalities and results to be used throughout the paper. More details can be found in \cite{Adams1978, EG2004}. 
We define the following spaces for the weak solution:
\begin{align*}
\bfX_f&=\{\bfv\in \bfH^1(\Omega_f):\, \bfv=\bfzero \mbox{ on }\Gamma_f^{ext} \},\\
\bfV_f&=\{\bfv\in \bfX_f:\, \nabla\cdot \bfv=0 \},\\
\bfX_p&=\{\bfxi\in \bfH^1(\Omega_p): \,\bfxi=\bfzero \mbox{ on }\Gamma_p^{s}, \,\bft_p^l\cdot \bfxi=0, \, 1\leq l\leq d-1 \mbox{ on }\Gamma_p^{ext} \},\\
Q_f&=L^2(\Omega_f),\\
Q_p&=\{r\in H^1(\Omega_p): \, r=0 \mbox{ on }\Gamma_p^{ext}\}.
\end{align*}
On these spaces, we have the following trace inequalities: % Theorem 5.22 Adams
\begin{alignat}{2}
&\forall \bfv\in \bfX_f, \quad\|\bfv\|_{\Gamma_I}&\leq&\, T_1|\bfv|_{1,\Omega_f} \quad \|\bfv\|_{\Gamma_f^{in}}\leq T_2|\bfv|_{1,\Omega_f}\label{trace1}\\
&\forall q\in Q_p,\quad \|q\|_{H^{1/2}(\Gamma_I)}&\leq&\, T_3|q|_{1,\Omega_p}, \quad \|q\|_{\Gamma_p^s}\leq T_4|q|_{1,\Omega_p}, \label{trace2}
\\&\forall \bfv\in \bfX_p, \quad\|\bfeta\|_{\Gamma_I}&\leq&\, T_{5}|\bfeta|_{1, \Omega_p},\label{trace3}
\end{alignat}
the Poincar\'{e} inequalities:
\begin{alignat}{2}
&\forall \bfv\in \bfX_f, \quad\|\bfv\|_{\Omega_f}&\leq&\, P_1|\bfv|_{1, \Omega_f}, \label{poincare1}\\
&\forall \bfeta\in \bfX_p, \quad\|\bfeta\|_{\Omega_p}&\leq&\, P_2|\bfeta|_{1, \Omega_p},\label{poincare2}\\
&\forall q\in Q_p,\quad \|q\|_{\Omega_p}&\leq&\, P_3|q|_{1, \Omega_p}, \label{poincare3}
\end{alignat}
a Sobolev inequality: 
\begin{equation}
\forall \bfv\in \bfX_f, \quad \|\bfv\|_{\bfL^4(\Omega_f)}\leq S_f|\bfv|_{\bfH^1(\Omega_f)}, \label{sobolev}
\end{equation}
and finally a Korn's inequality:
\begin{equation}
\forall \bfv\in \bfX_f, \quad |\bfv|_{1,\Omega_f}\leq K_f\|\bfD(\bfv)\|_{\Omega_f},\label{korn}
\end{equation}
where $T_1-T_5$, $P_1-P_3$, $S_f$ and $K_f$ are positive constants depending only on their corresponding domain.
\begin{theorem}[Gronwall's inequality in integral form]\cite{Gronwall}
Let
\begin{equation}\label{Gronwall}
\zeta(t)\leq B+C\int_0^t\zeta(s)ds, \mbox{ a.e. } t \in (0,T)
\end{equation}
where $\zeta$ is a continuous nonnegative function and $B, C\geq 0$ are constants.   Then
\[
\zeta(t)\leq B e^{Ct}.
\]
\end{theorem}
The next theorem is a compactness result that provides a strong convergence result which  is used to pass to the limit in the nonlinear terms of the Galerkin solution.
\begin{theorem}\cite[Corollary 4]{Simon}\label{Simon}
Let $X, B$ and $Y$ be Banach spaces such that $X\subset B\subset Y$ where the imbedding of $X$ into $B$ is compact. Let $F$ be a bounded set in $L^p(0,T;X)$ where $1\leq p< \infty$ and let the set $\{\frac{\partial f}{\partial t}\}_{f\in F}$ be bounded in $L^1(0,T;Y)$. Then $F$ is relatively compact in $L^p(0,T;B)$.
\end{theorem}

\section{Weak Formulation.}
In this section we derive the weak formulation of the problem. But first, we introduce additional notation and present assumptions on the problem data. We assume that $\bfK\in \bfL^{\infty}(\Omega_p)$ is independent of time, uniformly bounded and positive definite. There exists $K_{\min}, K_{\max}>0$ such that
\begin{equation}\label{KSPD}
\forall \bfx\in \overline{\Omega}_p, \quad K_{\min}\bfx\cdot \bfx\leq \bfK\bfx\cdot \bfx\leq K_{\max}\bfx\cdot \bfx.
\end{equation}
Further, we assume that $\bff_f\in \bfL^2(0,T;\bfL^2(\Omega_f))$, $\bff_s\in \bfL^2(0,T;\bfL^2(\Omega_p))$, $f_p\in L^2(0,T;L^2(\Omega_p))$ and $P_{in}\in L^2(0,T;H^{1/2}(\Gamma_f^{in}))$.
The weak formulation we propose for the problem is the following:\\
{\bf(WF1)} Find $\bfu_f\in L^{\infty}(0,T;\bfL^2(\Omega_f))\cap L^2(0,T;\bfX_f)$, $p_f\in L^1(0,T;Q_f)$, $\bfeta\in W^{1,\infty}(0,T;\bfL^2(\Omega_p))\times H^1(0,T;\bfX_p)$ and $p_p\in L^{\infty}(0,T;L^2(\Omega_p))\cap L^2(0,T;Q_p)$ where $\bfudot_f\in L^1(0,T;\bfL^{\frac32}(\Omega_f))$, 
such that for all $\bfv\in \bfX_f$, $q\in Q_f$, $\bfxi\in \bfX_p$ and $r\in Q_p$, 
\begin{align*}
&(\rho_f\dot{\bfu}_f,\bfv)_{\Omega_f}+(2\mu_f\bfD(\bfu_f),\bfD(\bfv))_{\Omega_f}+\rho_f(\bfu_f\cdot\nabla \bfu_f,\bfv)_{\Omega_f}-(p_f,\nabla\cdot \bfv)_{\Omega_f}\\
&
+(\rho_s\ddot{\bfeta},\bfxi)_{\Omega_p}+(2\mu_s\bfD(\bfeta),\bfD(\bfxi))_{\Omega_p}+(\lambda_s\nabla\cdot \bfeta,\nabla\cdot\bfxi)_{\Omega_p}-(\alpha p_p,\nabla\cdot \bfxi)_{\Omega_p}\\
&+(s_0\dot{p_p}+\alpha\nabla\cdot \dot{\bfeta},r)_{\Omega_p}+(\bfK\nabla p_p,\nabla r)_{\Omega_p}
\\
&+\langle p_p\bfn_\Gamma,\bfv-\bfxi\rangle_{\Gamma_I}+\Sigma_{l=1}^{d-1}\langle\beta(\bfu_f-\bfetadot)\cdot \bft_{\Gamma}^l,(\bfv-\bfxi)\cdot \bft_{\Gamma}^l\rangle_{\Gamma_I}+\langle(\bfetadot-\bfu_f)\cdot\bfn_{\Gamma},r\rangle_{\Gamma_I}\\
&\hspace{3cm}=-\langle P_{in}(t)\bfn_{\Gamma},\bfv\rangle_{\Gamma_f^{in}}+(\bff_f,\bfv)_{\Omega_f}+(\bff_s,\bfxi)_{\Omega_p}+(f_p,r)_{\Omega_p},\\
&(\nabla \cdot \bfu_f,q)_{\Omega_f}=0,
\end{align*}
a.e. in $(0,T)$, and 
\[
\bfu_f(0)=\bfzero \mbox{ a.e. in }\Omega_f, \quad \bfeta(0)=\bfzero,\quad \dot{\bfeta}(0)=\bfzero, \quad p_p(0)=0 \mbox{ a.e. in }\Omega_p.
\]
The reason for looking for a solution in $L^{\infty}(0,T;L^2(\Omega_i))$, $i=f, p$ may not seem obvious at this point, since typically the solutions are sought in $L^2(0,T;X)$ where $X$ is an appropriate Sobolev space, but we will prove that such a solution exists.
\subsection{Equivalence of the weak formulation {\bf{(WF1)}}.}
The following proposition establishes the equivalence between the coupled problem and the weak formulation {\bf{(WF1)}} proposed in the previous section.
\begin{proposition}\label{equivalence}
Let the data satisfy 
the assumptions listed in the previous section. Then each solution $\bfu_f\in L^{\infty}(0,T;\bfL^2(\Omega_f))\cap L^2(0,T;\bfX_f)$ such that $\bfudot_f\in L^1(0,T;\bfL^{\frac32}(\Omega_f))$, $p_f\in L^1(0,T;Q_f)$, $\bfeta\in W^{1,\infty}(0,T;\bfL^2(\Omega_p))\cap H^1(0,T;\bfX_p)$ 
 and $p_p\in  L^{\infty}(0,T;L^2(\Omega_p))\cap L^2(0,T;Q_p)$ of the problem defined by \eqref{moment}-\eqref{mass}, \eqref{structure}-\eqref{mass1} and \eqref{bc_f_in}-\eqref{inter4} is also a solution of the variational problem {\bf{(WF1)}} and conversely.
\end{proposition}
\begin{proof}
We first show sufficiency. To simplify the presentation, we included in  \ref{justification} the justification of using Green's formula in the following proof. Let $(\bfu_f, p_f, \bfeta, p_p)$ be a solution of the coupled problem defined by \eqref{moment}-\eqref{mass}, \eqref{structure}-\eqref{mass1}, \eqref{bc_f_in}-\eqref{inter4} satisfying the regularity stated in the proposition. We multiply \eqref{moment} by $\bfv\in \bfX_f$. After integration by parts:
\begin{multline*}
(\rho_f  \bfudot_f, \bfv)_{\Omega_f}
+(2  \mu_f\, D(\bfu_f),D(\bfv))_{\Omega_f}+\rho_f( \bfu_f\cdot\nabla\bfu_f,\bfv)_{\Omega_f}- (
p_f,\nabla \cdot\bfv)_{\Omega_f}\\
 -\langle \bfsigma_f\bfn_f, \bfv\rangle_{\partial\Omega_f} = (\bff_f, \bfv)_{\Omega_f}.
\end{multline*}
Using \eqref{bc_f_in}, \eqref{bc_f_out}, $\bfv=\bfzero$ on $\Gamma_f^{ext}$ and $\bfn_f=\bfn_{\Gamma}$ on $\Gamma_I$, we have
\begin{multline}\label{momentweak}
(\rho_f  \bfudot_f, \bfv)_{\Omega_f}
+(2  \mu_f\, D(\bfu_f),D(\bfv))_{\Omega_f}+\rho_f( \bfu_f\cdot\nabla\bfu_f,\bfv)_{\Omega_f}- (
p_f,\nabla \cdot\bfv)_{\Omega_f} \\-\langle \bfsigma_f\bfn_{\Gamma}, \bfv\rangle_{\Gamma_I} = (\bff_f, \bfv)_{\Omega_f}-\langle P_{in}(t), \bfv\rangle_{\Gamma_f^{in}}.
\end{multline}
Next we multiply \eqref{mass} by $q\in Q_f$ and integrate to get
\begin{equation}\label{massweak}
(\nabla\cdot \bfu_f,q)_{\Omega_f}=0.
\end{equation}
Multiplying \eqref{structure} by $\bfxi\in \bfX_p$ and integrating by parts yields:
\begin{multline*}
(\rho_p \bfetaddot,\bfxi)_{\Omega_p}
+(2 \mu_s \,  D(\bfeta), D(\bfxi))_{\Omega_p} + (\lambda_s (\nabla \cdot \bfeta), \nabla\cdot \bfxi)_{\Omega_p}  -(\alpha  p_p,\nabla\cdot\bfxi)_{\Omega_p}\\-\langle\bfsigma_p\bfn_p,\bfxi \rangle_{\partial\Omega_p}
= (\bff_s,\bfxi)_{\Omega_p}.
\end{multline*}
Observe that \eqref{bc_p_ext} implies $\bfsigma_p\bfn_p=\bfsigma_p^E\bfn_p=\Sigma_{l=1}^{d-1}(\bft_p^l\cdot\bfsigma_p^E\bfn_p)\cdot \bft_p^l$ on $\Gamma_p^{ext}$. Then since $\bft_p^l\cdot\bfxi=0$ on $\Gamma_p^{ext}$,  $\bfxi=0$ on $\Gamma_p^s$ and $\bfn_p=-\bfn_{\Gamma}$ on $\Gamma_I$, we have
\begin{multline}\label{structureweak}
(\rho_p  \bfetaddot,\bfxi)_{\Omega_p}
+(2 \mu_s \,  D(\bfeta), D(\bfxi))_{\Omega_p} + (\lambda_s (\nabla \cdot \bfeta), \nabla\cdot \bfxi)_{\Omega_p} - (\alpha  p_p,\nabla\cdot\bfxi)_{\Omega_p}\\
+\langle\bfsigma_p\bfn_{\Gamma},\bfxi \rangle_{\Gamma_I}= (\bff_s,\bfxi)_{\Omega_p}.
\end{multline}
Multiplying \eqref{mass1} by $r\in Q_p$ and integrating over $\Omega_p$, we obtain
\begin{align*}
(s_0 \pdot_p+ \alpha  \nabla \cdot \bfetadot,r)_{\Omega_p} +(\bfK\nabla p_p, \nabla r)_{\Omega_p}
-\langle\bfK\nabla p_p\cdot \bfn_p, r\rangle_{\partial\Omega_p}= (f_p,r)_{\Omega_p}.
\end{align*}
Using \eqref{bc_p_s}, $r=0$ on $\Gamma_p^{ext}$, and $\bfn_p=-\bfn_{\Gamma}$ on $\Gamma_I$, we have
\begin{equation}\label{mass1weak}
(s_0 \pdot_p+ \alpha  \nabla \cdot \bfetadot,r)_{\Omega_p} +(\bfK\nabla p_p, \nabla r)_{\Omega_p}
+\langle\bfK\nabla p_p\cdot \bfn_{\Gamma}, r\rangle_{\Gamma_I}= (f_p,r)_{\Omega_p}.
\end{equation}
Next, we rewrite the interface integrals using the interface conditions \eqref{inter1}-\eqref{inter4}.
On $\Gamma_I$, by \eqref{inter3} and \eqref{inter4}, we have
\begin{equation}\label{normalstress}
\bfsigma_f\bfn_{\Gamma}=(\bfn_{\Gamma}\cdot \bfsigma_f\bfn_{\Gamma})\bfn_{\Gamma}+\Sigma_{l=1}^{d-1}(\bft_{\Gamma}^l\cdot \bfsigma_f\bfn_{\Gamma})\bft_{\Gamma}^l= -p_p\bfn_{\Gamma}-\Sigma_{l=1}^{d-1}(\beta(\bfu_f-\bfetadot)\cdot \bft_{\Gamma}^l)\bft_{\Gamma}^l.
\end{equation}
(Note here that since $\sigma_f$ is symmetric, $\bft_{\Gamma}^l\cdot \bfsigma_f\bfn_{\Gamma}=\bfn_{\Gamma}\cdot \bfsigma_f\bft_{\Gamma}^l$.)
Adding \eqref{momentweak}, \eqref{massweak}, \eqref{structureweak} and \eqref{mass1weak} while using \eqref{normalstress} for $\langle \bfsigma_f\bfn_{\Gamma},\bfv \rangle_{\Gamma_{I}}$, \eqref{inter1} for $\langle\bfK\nabla p_p\cdot \bfn_{\Gamma}, r\rangle_{\Gamma_I}$ and \eqref{inter2} and \eqref{normalstress} for $\langle \bfsigma_p\bfn_{\Gamma},\bfxi \rangle_{\Gamma_{I}}$ gives  the weak formulation {\bf{(WF1)}}.

For the converse, let $(\bfu_f, p_f, \bfeta, p_p)$ be a solution of {\bf{(WF1)}}. 
We pick first $\bfv\in \bfcalD(\Omega_f)$, $r=0$ and $\bfxi=\bfzero$, second $\bfv\in \bfzero$, $r\in D(\Omega_p)$ and $\bfxi=\bfzero$ and last $\bfv=\bfzero$, $r=0$ and $\bfxi\in \bfcalD(\Omega_p)$. This gives \eqref{moment} on $\Omega_f$ and \eqref{structure} and \eqref{mass1} on $\Omega_p$ in the sense of distributions. Next we multiply \eqref{moment} with $\bfv\in \bfX_f$, \eqref{structure} with $\bfxi\in \bfX_p$ and \eqref{mass1} with $r\in Q_p$ and apply Green's formulas and add the outcomes to get
\begin{align*}
(\rho_f\dot{\bfu}_f&,\bfv)_{\Omega_f}+(2\mu_f\bfD(\bfu_f),\bfD(\bfv))_{\Omega_f}+\rho_f(\bfu_f\cdot\nabla \bfu_f,\bfv)_{\Omega_f}-(p_f,\nabla\cdot \bfv)_{\Omega_f}\\
&+(\rho_s\ddot{\bfeta},\bfxi)_{\Omega_p}+(2\mu_s\bfD(\bfeta),\bfD(\bfxi))_{\Omega_p}+(\lambda_s\nabla\cdot \bfeta,\nabla\cdot\bfxi)_{\Omega_p}-(\alpha p_p,\nabla\cdot \bfxi)_{\Omega_p}\\
&+(s_0\dot{p_p}+\alpha\nabla\cdot \dot{\bfeta},r)_{\Omega_p}+(\bfK\nabla p_p,\nabla r)_{\Omega_p}
\\&-\langle \bfsigma_f
\bfn_{f},\bfv\rangle_{\Gamma_f^{in}\cup\Gamma_f^{out}\cup\Gamma_I}-\langle \bfsigma_p
\bfn_p, \bfxi\rangle_{\Gamma_p^{ext}\cup\Gamma_I}-\langle \bfK \nabla p_p\cdot \bfn_p,r \rangle_{\Gamma_{p}^s\cup\Gamma_I}\\
&=(\bff_f,\bfv)_{\Omega_f}+(\bff_s,\bfxi)_{\Omega_p}+(f_p,r)_{\Omega_p}.
\end{align*}
Comparing this with {\bf{(WF1)}} gives
\begin{align}
\langle p_p\bfn_\Gamma,\bfv-\bfxi\rangle_{\Gamma_I}&+\Sigma_{l=1}^{d-1}\langle\beta(\bfu_f-\bfetadot)\cdot \bft_{\Gamma}^l,(\bfv-\bfxi)\cdot \bft_{\Gamma}^l\rangle_{\Gamma_I}\nonumber\\
&\hspace{2cm}+\langle(\bfetadot-\bfu_f)\cdot\bfn_{\Gamma},r\rangle_{\Gamma_I}
+\langle P_{in}(t)\bfn_f,\bfv\rangle_{\Gamma_f^{in}}
\label{aux1}\\
=-&\langle \bfsigma_f
\bfn_{f},\bfv\rangle_{\Gamma_f^{in}\cup\Gamma_f^{out}\cup\Gamma_I}-\langle \bfsigma_p
\bfn_p, \bfxi\rangle_{\Gamma_p^{ext}\cup\Gamma_I}-\langle \bfK \nabla p_p\cdot \bfn_p,r \rangle_{\Gamma_{p}^s\cup\Gamma_I}\nonumber
\end{align}
for all $(\bfv, \bfxi, r)\in \bfX_f\times \bfX_p\times Q_p$.
If we let $\bfv=\bfzero$, $\bfxi=\bfzero$ in \eqref{aux1}, we get 
\begin{equation}\label{aux2}
-\langle \bfK \nabla p_p\cdot \bfn_p,r \rangle_{\Gamma_{p}^s\cup\Gamma_I}=\langle(\bfetadot-\bfu_f)\cdot\bfn_{\Gamma},r\rangle_{\Gamma_I}, \forall r\in Q_p.
\end{equation}
The choice $r\in Q_p$ such that $r|_{\Gamma_I}=0$ yields
\[
\langle \bfK \nabla p_p\cdot \bfn_p,r \rangle_{\Gamma_p^s}=0
\]
which implies the first condition of \eqref{bc_p_s}.
Using this  in \eqref{aux2}, we get 
\[
-\langle \bfK \nabla p_p\cdot \bfn_p,r\rangle_{\Gamma_I}=\langle(\bfetadot-\bfu_f)\cdot\bfn_{\Gamma},r\rangle_{\Gamma_I}, \forall r\in Q_p
\]
 which yields \eqref{inter1}. This reduces 
\eqref{aux1} to 
\begin{align}
&\langle p_p\bfn_\Gamma,\bfv-\bfxi\rangle_{\Gamma_I}+\sum_{l=1}^{d-1}\langle\beta(\bfu_f-\bfetadot)\cdot \bft_{\Gamma}^l,(\bfv-\bfxi)\cdot \bft_{\Gamma}^l\rangle_{\Gamma_I}+\langle P_{in}(t)\bfn_f,\bfv\rangle_{\Gamma_f^{in}}\label{aux3}\\
&=-\langle \bfsigma_f
\bfn_{f},\bfv\rangle_{\Gamma_f^{in}\cup\Gamma_f^{out}\cup\Gamma_I}-\langle \bfsigma_p
\bfn_p, \bfxi\rangle_{\Gamma_p^{ext}\cup\Gamma_I}, \forall \bfv\in \bfX_f, \forall \bfxi\in \bfX_p.\nonumber
\end{align}
Now we let $\bfv=\bfzero$ in \eqref{aux3}. Then
\[
-\langle p_p\bfn_{\Gamma},\bfxi\rangle_{\Gamma_I}-\sum_{l=1}^{d-1}\langle\beta(\bfu_f-\bfetadot)\cdot \bft_{\Gamma}^l,\bfxi\cdot \bft_{\Gamma}^l\rangle_{\Gamma_I}\\
=-\langle \bfsigma_p\bfn_p, \bfxi\rangle_{\Gamma_p^{ext}\cup\Gamma_I},  \forall \bfxi\in \bfX_p.
\]
Since $p_p=0$ on $\Gamma_p^{ext}$, the choice $\bfxi\in \bfX_p$  %($\bft_p^l\cdot\bfxi=0$) 
such that $\bfxi=\bfzero$ on $\Gamma_I$ implies
\[
0
=\langle\bfsigma_p\bfn_p, \bfxi\rangle_{\Gamma_{ext}^p}=\langle\bfn_p\cdot\bfsigma_p\bfn_p, \bfxi\cdot \bfn_p\rangle_{\Gamma_{ext}^p}=\langle\bfn_p\cdot\bfsigma_p^E\bfn_p, \bfxi\cdot \bfn_p\rangle_{\Gamma_{ext}^p}.
\]
Therefore we recover %$\bfn_p\cdot\bfsigma_p^E\bfn_p=\bfzero$ on $\Gamma_p^{ext}\times (0,T)$ 
in the sense of distributions the second condition in \eqref{bc_p_ext}. This also yields
\[
\langle -p_p\bfn_{\Gamma}-\sum_{l=1}^{d-1}\beta((\bfu_f-\bfetadot)\cdot \bft_{\Gamma}^l)\bft_{\Gamma}^l,\bfxi \rangle_{\Gamma_I}\\
=-\langle \bfsigma_p\bfn_{p}, \bfxi\rangle_{\Gamma_I},  \forall \bfxi\in \bfX_p.
\]
Therefore 
\begin{equation}\label{aux4}
-p_p\bfn_{\Gamma}-\sum_{l=1}^{d-1}\beta((\bfu_f-\bfetadot)\cdot \bft_{\Gamma}^l)\bft_{\Gamma}^l=\bfsigma_p\bfn_{\Gamma}
\end{equation}
holds in the sense of distributions.
This reduces \eqref{aux3} to 
\begin{multline*}
\langle p_p\bfn_\Gamma,\bfv\rangle_{\Gamma_I}+\sum_{l=1}^{d-1}\langle\beta(\bfu_f-\bfetadot)\cdot \bft_{\Gamma}^l,\bfv\cdot \bft_{\Gamma}^l\rangle_{\Gamma_I}+\langle P_{in}(t)\bfn_f,\bfv\rangle_{\Gamma_f^{in}}\\
=-\langle \bfsigma_f\bfn_{f},\bfv\rangle_{\partial\Omega_f}, \forall \bfv\in \bfX_f.
\end{multline*}
Letting $\bfv\in \bfX_f$ such that $\bfv=\bfzero$ on $\Gamma_I$ gives
\[
\langle P_{in}(t)\bfn_f,\bfv\rangle_{\Gamma_f^{in}}=-\langle \bfsigma_f\bfn_{f},\bfv\rangle_{\Gamma_f^{in}\cup\Gamma_f^{out}}.
\]
Picking $\bfv$ such that $\bfv=\bfzero$ on $\Gamma_f^{out}$ implies %$-P_{in}(t)\bfn_f=\bfsigma_f\bfn_{f}$ on $\Gamma_f^{in}\times (0,T)$, which is 
\eqref{bc_f_in}. If we plug this in the above equation we get $0=-\langle \bfsigma_f\bfn_{f},\bfv\rangle_{\Gamma_f^{out}}$ concluding that $\bfsigma_f\bfn_{f}=\bfzero$ on $\Gamma_f^{out}\times (0,T)$ in the distributional sense. This gives \eqref{bc_f_out} and also implies that 
\[
\langle p_p\bfn_\Gamma,\bfv\rangle_{\Gamma_I}+\sum_{l=1}^{d-1}\langle\beta(\bfu_f-\bfetadot)\cdot \bft_{\Gamma}^l,\bfv\cdot \bft_{\Gamma}^l\rangle_{\Gamma_I}=-\langle \bfsigma_f\bfn_{f},\bfv\rangle_{\Gamma_I}, \forall \bfv\in \bfX_f.
\]
Therefore 
\[
p_p\bfn_{\Gamma}+\sum_{l=1}^{d-1}\langle(\beta(\bfu_f-\bfetadot)\cdot \bft_{\Gamma}^l)\bft_{\Gamma}^l=-\bfsigma_f\bfn_{f}.
\] 
This compared to \eqref{aux4} implies \eqref{inter2} and also gives \eqref{inter3} and \eqref{inter4} after dotted with $\bfn_{\Gamma}$ and $\bft_{\Gamma}^l$, $1\leq j\leq d-1$. 
\end{proof}
\section{Main results.}
This section summarizes the main results of this paper. First, for the sake of simplicity, we define the following functions of time:
\begin{subequations}\label{Cs}
\begin{align}
\calC_1(t)&=\Big(\dfrac{3 T_2^2K_f^2}{4\mu_f}\|P_{in}(t)\|^2_{\Gamma_{f}^{in}}+\dfrac{ 3P_1^2K_f^2}{4\mu_f}\|\bff_f(t)\|^2_{\Omega_f}+\dfrac{P_3^2}{2K_{\min}}\|f_p(t)\|_{\Omega_p}^2+\dfrac12\|\bff_s(t)\|_{\Omega_p}^2\Big)^{1/2},\\
\calC_2(t)&=\Big(\dfrac{3 T_2^2K_f^2}{2\mu_f}\|\dot{P}_{in}(t)\|_{\Gamma_f^{in}}^2+\dfrac{3 P_1^2K_f^2}{2\mu_f}\|\dot{\bff_f}(t)\|^2_{\Omega_f}+\dfrac{P_3^2}{2K_{\min}}\|\dot{f_p}(t)\|_{\Omega_p}^2+\dfrac12\|\dot{\bff}_s(t)\|^2_{\Omega_p}\Big)^{1/2},
\end{align}
a.e. in $(0,T)$,
and the following constant:
\begin{align}
\calC_3=&\left(\dfrac{C_j^2}{\rho_f}\|P_{in}(0)\|^2_{H_{00}^{1/2}(\Gamma_f^{in})}+\dfrac1{\rho_f}\|\bff_f(0)\|_{\Omega_f}^2
+\dfrac1{2s_0}\|f_p(0)\|^2_{\Omega_p}+\dfrac1{2\rho_s}\|\bff_s(0)\|^2_{\Omega_p}\right)^{1/2}, 
\end{align}
\end{subequations}
where the constants $T_2, K_f, P_1, P_3$ are defined in Section \ref{usefulresults} and $C_j$ is the continuity constant of the continuous lifting operator from $H^{1/2}(\partial \Omega_f)\rightarrow H^1(\Omega_f)$. 
Observe that $\calC_1, \calC_2$ and $\calC_3$ depend only on the data of the  problem. 
We now present out main existence and uniqueness result.
\begin{theorem}\label{mainresult}
Assume that $\bff_f\in H^1(0,T;\bfL^2(\Omega_f))$, $\bff_s\in H^1(0,T;\bfL^2(\Omega_p))$, $\bff_p\in H^1(0,T;L^2(\Omega_f))$ and $P_{in}\in H^1(0,T;H^{1/2}(\Gamma_f^{in}))$ and that the following small data condition holds:
\begin{multline}
(1+\dfrac{T}{\rho_s}e^{T/\rho_s})\|\calC_2\|_{L^2(0,T)}^2+\dfrac{T}{\rho_s^2}e^{T/\rho_s}\|\bff_s(0)\|^2_{\Omega_p}\label{smalldatacond}\\+(1+\dfrac1{\rho_s}e^{T/\rho_s}+\dfrac{T}{\rho_s^2}e^{T/\rho_s})\|\calC_1\|_{L^2(0,T)}^2+\|\calC_1\|^2_{L^{\infty}(0,T)}<\dfrac{\mu_f^3}{9\rho_f^2S_f^4K_f^6}.
\end{multline}
Then, problem {\bf{(WF1)}}has a unique solution $(\bfu_f, p_f,\bfeta, p_p)$ 
such that 
\begin{align}
&\dfrac{\rho_f}2\|\bfu_f\|_{L^{\infty}(0,T;\bfL^2(\Omega_f))}^2
+\frac{\rho_s}2\|\dot{\bfeta}\|_{L^{\infty}(0,T;\bfL^2(\Omega_p))}^2+\mu_s\|\bfD(\bfeta)\|^2_{L^{\infty}(0,T;\Lbb^2(\Omega_p))}\nonumber\\
&+\frac{s_0}2\|p_p\|_{L^{\infty}(0,T;L^2(\Omega_f))}^2
+\mu_f\|\bfD(\bfu_f)\|^2_{L^2(0,T;\Lbb^2(\Omega_f))}+\dfrac12\|\bfK^{1/2}\nabla p_p\|^2_{L^2(0,T;\bfL^2(\Omega_p))}\nonumber\\
&\hspace{5cm}\leq \left(1+\dfrac{T}{\rho_s}e^{T/\rho_s}\right)\|\calC_1\|^2_{L^2(0,T)}.\label{ineq:mainresult1}
\end{align}
Furthermore,
\begin{align}
\|\bfD(\bfu_f)\|_{L^{\infty}(0,T;\Lbb^2(\Omega_f))}&<\dfrac{\mu_f}{3\rho_fS_f^2K_f^3},\label{ineq:mainresult2}
\end{align}
\begin{align}
&\dfrac{\rho_f}2\|\dot{\bfu}_f\|^2_{L^{\infty}(0,T;\bfL^2(\Omega_f))}
+\dfrac{\rho_s}2\|\ddot{\bfeta}\|^2_{L^{\infty}(0,T;\bfL^2(\Omega_p))}+\mu_s\|\bfD(\dot{\bfeta})\|^2_{L^{\infty}(0,T;\Lbb^2(\Omega_p))}\nonumber\\
&+\dfrac{s_0}2\|\dot{p}_p\|^2_{L^{\infty}(0,T;L^2(\Omega_f))}+\mu_f\|\bfD(\dot{\bfu}_f)\|^2_{L^2(0,T;L^2(\Omega_f))}+\dfrac12\|\bfK^{1/2}\nabla \dot{p}_p\|^2_{L^2(0,T;\bfL^2(\Omega_p))}\nonumber
\\
&\hspace{2cm}\leq (1+\dfrac{T}{\rho_s}e^{2T/\rho_s^2})\|\calC_2\|^2_{L^2(0,T)}+\dfrac{Te^{2T/\rho_s^2}}2\calC_3,\label{ineq:mainresult3}
\end{align}
and 
\begin{align}\label{pfbound}
&\|p_f\|_{L^{\infty}(0,T;L^2(\Omega_f))}\leq \dfrac1{\kappa}\big(\rho_f\|\dot{\bfu}_f\|_{L^{\infty}(0,T;\bfL^2(\Omega_f))}+2\mu_f\|\bfD(\bfu_f)\|_{L^{\infty}(0,T;\Lbb^2(\Omega_f))}\nonumber\\
&+S_f^2\|\bfu_f\|_{L^{\infty}(0,T;\bfH^1(\Omega_f))}^2
+T_1T_3\|p_p\|_{L^{\infty}(0,T;H^1(\Omega_p))}+\beta T_1^2\|\bfu_f\|_{L^{\infty}(0,T;\bfH^1(\Omega_f))}\nonumber\\
&+\beta T_1T_5\|\dot{\bfeta}\|_{L^{\infty}(0,T;\bfH^1(\Omega_p))}+T_2\|P_{in}\|_{L^{\infty}(0,T;L^2(\Gamma_f^{in}))}+\|\bff_f\|_{L^{\infty}(0,T;\bfL^2(\Omega_f))}\big).
\end{align}
\end{theorem}
The constants $T_1-T_5, K_f, S_f, P_1, P_3$ used in the above estimates are defined in Section \ref{usefulresults}.
\section{Proof of Theorem \ref{mainresult}.}
The proof consists of multiple steps. The main idea is to use Galerkin's method on the divergence-free version of the weak problem {\bf{(WF1)}} in which the fluid pressure $p_f$ is eliminated. We will first present the divergence-free formulation {\bf (WF2)} and introduce its Galerkin approximation {\bf{(GF)}}. Next we prove that there exists a unique maximal Galerkin solution by writing {\bf{(GF)}} as a system of first order equations and applying the theory of ordinary differential equations. However this existence result holds only on a finite subiniterval of $[0,T]$. Demonstrating a priori bounds for the Galerkin solution guarantees the validity of this existence result on the entire interval $[0,T]$ and also allows us to pass to the limit. At the end of this process, we obtain a solution $\bfu_f, \bfeta$ and $p_p$ of the divergence-free weak formulation. We conclude the proof using an inf-sup condition to recover the fluid pressure $p_f$ that was eliminated from the weak formulation, proving the equivalence of {\bf{(WF2)}} and {\bf{(WF1)}}. This last step also provides a priori estimates for $p_f$.

\subsection{A divergence-free weak formulation.}
For the analysis of the problem, we will focus on the following divergence free version of the formulation {\bf{(WF1)}}.\\\\
{\bf (WF2)} Find $\bfu_f\in L^{\infty}(0,T;\bfL^2(\Omega_f))\cap L^2(0,T;\bfX_f)$, $\bfeta\in W^{1,\infty}(0,T;\bfL^2(\Omega_p))\times H^1(0,T;\bfX_p)$ and $p_p\in L^{\infty}(0,T;L^2(\Omega_p))\cap L^2(0,T;Q_p)$  with \\$\bfudot_f\in L^1(0,T;\bfL^{\frac32}(\Omega_f))$  such that for all $\bfv\in \bfV_f$, $\bfxi\in \bfX_p$ and $r\in Q_p$, 
\begin{align*}
&(\rho_f\dot{\bfu}_f,\bfv)_{\Omega_f}+(2\mu_f\bfD(\bfu_f),\bfD(\bfv))_{\Omega_f}+\rho_f(\bfu_f\cdot\nabla \bfu_f,\bfv)_{\Omega_f}
\\
&+(\rho_s\ddot{\bfeta},\bfxi)_{\Omega_p}+(2\mu_s\bfD(\bfeta),\bfD(\bfxi))_{\Omega_p}+(\lambda_s\nabla\cdot \bfeta,\nabla\cdot\bfxi)_{\Omega_p}-(\alpha p_p,\nabla\cdot \bfxi)_{\Omega_p}\\
&+(s_0\dot{p_p}+\alpha\nabla\cdot \dot{\bfeta},r)_{\Omega_p}+(\bfK\nabla p_p,\nabla r)_{\Omega_p}
\\
&+\langle p_p\bfn_\Gamma,\bfv-\bfxi\rangle_{\Gamma_I}+\Sigma_{l=1}^{d-1}\langle\beta(\bfu_f-\bfetadot)\cdot \bft_{\Gamma}^l,(\bfv-\bfxi)\cdot \bft_{\Gamma}^l\rangle_{\Gamma_I}+\langle(\bfetadot-\bfu_f)\cdot\bfn_{\Gamma},r\rangle_{\Gamma_I}\\
&\hspace{3.0cm}=-\langle P_{in}(t)\bfn_{\Gamma},\bfv\rangle_{\Gamma_f^{in}}+(\bff_f,\bfv)_{\Omega_f}+(\bff_s,\bfxi)_{\Omega_p}+(f_p,r)_{\Omega_p}
\end{align*}
a.e. in $(0,T)$, and 
\[
\bfu_f(0)=\bfzero \mbox{ a.e. in }\Omega_f, \quad \bfeta(0)=\bfzero,\quad \dot{\bfeta}(0)=\bfzero, \quad p_p(0)=0 \mbox{ a.e. in }\Omega_p.
\]
Note that the unknown pressure $p_f$ is no longer in the weak formulation. Furthermore, it is obvious that any solution of {\bf{(WF1)}} is a solution of {\bf{(WF2)}}. The converse will be proved in Section \ref{sec:NSEpressure} using an inf-sup condition.
\subsection{A semi-discrete Galerkin formulation of {\bf{(WF2)}}.}
The existence result is proved by constructing a sequence of approximate problems and then passing to the limit, that is, using the Galerkin method. 
Separability of $\bfV_f\times \bfX_p\times Q_p$ implies the existence of a basis $\{(\bfv_i, \bfxi_i, r_i)\}_{i\geq 0}$ consisting of smooth functions. 
We define 
\begin{align*}
\bfV_f^m&= \rmspan\{\bfv_i:i=1,\hdots,m\}, \\\bfX_p^m&=\rmspan\{\bfxi_i:i=1,\hdots,m\}, \\ Q_p^m&=\rmspan\{r_i:i=1,\hdots,m\},\end{align*}
and use the following Galerkin approximations for the unknowns $\bfu_f, \bfeta$ and $p_p$:
\begin{equation}\label{galerkinexpansion}
\bfu_m(\bfx,t)=\sum_{j=1}^m \alpha_j(t)\bfv_j(\bfx), \bfeta_m(\bfx,t)=\sum_{j=1}^m \beta_j(t)\bfxi_j(\bfx), p_m(\bfx,t)=\sum_{j=1}^m \gamma_j(t)r_j(\bfx). 
\end{equation}
Then, we can write the Galerkin approximation of the problem (WF2) as follows: \\
{\bf (GF)}  Find $\bfu_m\in \calC^{1}(0,T;\bfV_f^m)$, $\bfeta_m\in \calC^2(0,T;\bfX_p^m)$, $p_m\in \calC^1(0,T;Q_p^m)$ such that
\begin{align*}
(\rho_f\bfudot_m,\bfv)_{\Omega_f}&+(2\mu_f\bfD(\bfu_m),\bfD(\bfv))_{\Omega_f}+\rho_f(\bfu_m\cdot\nabla \bfu_m,\bfv)_{\Omega_f}+(\rho_s\ddot{\bfeta}_m,\bfxi)_{\Omega_p}\\
&
+(2\mu_s\bfD(\bfeta_m),\bfD(\bfxi))_{\Omega_p}
+(\lambda_s\nabla\cdot \bfeta_m,\nabla\cdot\bfxi)_{\Omega_p}-(\alpha p_m,\nabla\cdot \bfxi)_{\Omega_p}\\
&+(s_0\dot{p}_m+\alpha\nabla\cdot \bfetadot_m,r)_{\Omega_p}+(\bfK\nabla p_m,\nabla r)_{\Omega_p}+\langle p_m\bfn_\Gamma,\bfv-\bfxi\rangle_{\Gamma_I}
\\
&+\Sigma_{l=1}^{d-1}\langle\beta(\bfu_m-\bfetadot_m)\cdot \bft_{\Gamma}^l,(\bfv-\bfxi)\cdot \bft_{\Gamma}^l\rangle_{\Gamma_I}+\langle(\bfetadot_m-\bfu_m)\cdot\bfn_{\Gamma},r\rangle_{\Gamma_I}\\
&\hspace{1.8cm}=-\langle P_{in}(t)\bfn_{\Gamma},\bfv\rangle_{\Gamma_f^{in}}+(\bff_f,\bfv)_{\Omega_f}+(\bff_s,\bfxi)_{\Omega_p}+(f_p,r)_{\Omega_p},
\end{align*}
for all $(\bfv, \bfxi, r)\in \bfV_f^m\times \bfX_p^m\times Q_p^m$, a.e. $t \in (0,T)$
and 
\begin{equation}\label{GalerkinIC}
\bfu_m(0)=\bfzero, \quad \bfeta_m(0)=\bfzero, \quad \bfetadot_m(0)=\bfzero,\quad  p_m(0)=0.
\end{equation}
\begin{lemma}
For each positive integer $m$, the formulation {\bf (GF)}  has a unique maximal solution $(\bfu_m,\bfeta_m,p_m)\in \calC^1(0,T_m;\bfV_f^m)\times \calC^2(0,T_m;\bfX_p^m)\times \calC^1(0,T_m;Q_p^m)$ for some time $T_m$ where $0<T_m\leq T$.
\end{lemma}
\begin{proof}
Using the Galerkin expansions given in \eqref{galerkinexpansion}, the problem {(\bf GF)} can be represented in matrix form. The following is a standard finite-dimensional argument which is basically defining the problem as a square first order system of ordinary differential equations (ODE) with an initial condition.
For the integrals on the left hand side for $1\leq i, j\leq m$, we define 
\[
\bfA^f_{ij}=\rho_f(\bfv_j,\bfv_i)_{\Omega_f}, \quad \bfB_{ij}^f=2\mu_f(\bfD(\bfv_j),\bfD(\bfv_i))_{\Omega_f}+\Sigma_{l=1}^d\beta\langle \bfv_j\cdot \bft_{\Gamma}^l, \bfv_i\cdot \bft_{\Gamma}^l \rangle_{\Gamma_I}, 
\]
\[
\bfN_i=(\rho_f(\bfv_j\cdot \nabla \bfv_k, \bfv_i)_{\Omega_f})_{1\leq j,k\leq m},
\]
\[
\bfA^s_{ij}=\rho_s(\bfxi_j,\bfxi_i)_{\Omega_p},\quad \bfB_{ij}^s=2\mu_s(\bfD(\bfxi_j),\bfD(\bfxi_i))_{\Omega_p}+\lambda_s(\nabla \cdot\bfxi_j,\nabla \cdot\bfxi_i)_{\Omega_p}, 
\]
\[
\bfA^p_{ij}=s_0(r_j,r_i)_{\Omega_p},\quad \bfB_{ij}^p=(\bfK\nabla r_j,\nabla r_i)_{\Omega_p}, 
\]
\[
\bfC_{ij}=\alpha(r_j,\nabla \cdot\bfxi_i)_{\Omega_p}+\langle r_j\bfn_{\Gamma},\bfxi_i\rangle_{\Gamma_I},\quad \bfD_{ij}=\langle r_j\bfn_{\Gamma},\bfv_i\rangle_{\Gamma_I},
\]
\[
\bfE_{ij}=\Sigma_{l=1}^d\beta\langle \bfxi_j\cdot \bft_{\Gamma}^l, \bfv_i\cdot \bft_{\Gamma}^l \rangle_{\Gamma_I}, \bfF_{ij}=\Sigma_{l=1}^d\beta\langle \bfxi_j\cdot \bft_{\Gamma}^l, \bfxi_i\cdot \bft_{\Gamma}^l \rangle_{\Gamma_I}.
\]
And finally for the right hand side integrals we define
\[
\bfa_i=-\langle P_{in}(t)\bfn_f,\bfv_i\rangle_{\Omega_f}+(f_f,\bfv_i)_{\Omega_f},\quad \bfb_i=(\bff_s,\bfxi_i)_{\Omega_p}, \quad \bfc_i=(f_p,r_i)_{\Omega_p}.
\]
The unknowns are $\bfalpha_i=\alpha_i(t)$, $\bfbeta_i=\beta_i(t)$ and $\bfgamma_i=\gamma_i(t)$, $i=1,\hdots, m$ and we define a vector that holds these unknowns and $\bftheta = \dot{\bfbeta}$ as follows :
\[
\bfw(t)=\left[
\begin{array}{c}
\bfalpha(t)\\\bfbeta(t)\\\bfgamma(t)\\\bftheta(t)
\end{array}
\right]
\]
and set $(\calN(\bfw))_i=\bfN_i\bfalpha\cdot \bfalpha$.
With these definitions, ${(\bf GF)}$ is equivalent to finding $\bfalpha, \bfbeta, \bfgamma$ such that
\begin{align*}
\bfA^f\dot{\bfalpha}+\bfB^f\bfalpha+\calN(\bfw)+\bfD\bfgamma-\bfE\dot{\bfbeta}=\bfa
\\
\bfA^s\ddot{\bfbeta}+\bfB^s\bfbeta-\bfC\bfgamma-\bfE^T\bfalpha+\bfF\dot{\bfbeta}=\bfb\\
\bfA^p\dot{\bfgamma}+\bfB^p\bfgamma+\bfC^{T}\dot{\bfbeta}-\bfD^T\bfalpha=\bfc
\end{align*}
where $
\bfalpha(0),  \bfbeta(0)$ and $\bfgamma(0)$ are given.
We can rewrite this as a system of first order equations as follows:
\[
\bfM\dot{\bfw}+\bfN\bfw=\bfd(\bfw),
\]
where
\[
\bfM=\left[
\begin{array}{cccc}
\bfA^f & \bfzero & \bfzero&\bfzero\\
\bfzero & \bfI& \bfzero&\bfzero\\
\bfzero & \bfzero & \bfA^p&\bfzero\\
\bfzero & \bfzero & \bfzero&\bfA^s\\
\end{array}
\right],\quad \bfN=\left[
\begin{array}{cccc}
\bfB^f& \bfzero & \bfD&-\bfE\\
\bfzero  & \bfzero & \bfzero &-\bfI\\
-\bfD^T & \bfzero & \bfB^p&\bfC^T\\
-\bfE^T & \bfB^s & -\bfC&\bfzero\\
\end{array}
\right]
\]
and 
\[
 \bfd(\bfw)=\left[\begin{array}{c}
\bfa-\calN(\bfw)\\\bfzero\\\bfc\\\bfb
\end{array}\right].
\]
Since $\rho_f, \rho_s$ and $s_0$ are positive, $\bfA^f, \bfA^p$ and $\bfA^s$ are symmetric positive definite implying that $\bfM$ is invertible. This defines an autonomous ODE in 
$\bfw(t)$ such that 
\[
\dot{\bfw}=\bfM^{-1}\bfd(\bfw)-\bfM^{-1}\bfN\bfw=:g(\bfw), \quad \mbox{ where } \bfw(0) \mbox{ is given.}
\]
The matrices $\bfM, \bfN$ are $4m\times 4m$ and the  vectors $\bfd, \bfw$ have length $4m$.
It is obvious that the function $g$ is continuous in time and locally Lipschitz continuous in $\bfw$. Then, it follows from the theory of ordinary differential equations \cite{CoddingtonLevinson} that there is a unique maximal solution $\bfw$ in the interval $[0,T_m]$ for some $T_m$ such that $0<T_m\leq T$ such that each component of $\bfw$, i.e., each component of $\bfalpha, \bfbeta, \bfgamma$ and $\bftheta=\dot{\bfbeta}$ belongs to $\calC^1(0,T_m).$ 
\end{proof}
We need a priori bounds on the Galerkin solution to conclude that $T_m=T$. We discuss this next in Section \ref{sec:apriori}. %(See \cite[p.32]{CesmeliogluThesis}).
\begin{remark}
Note that if we consider the Stokes problem for the fluid part, so if there is no nonlinearity, an existence and uniqueness result will be global on $[0,T]$.
\end{remark}
\subsection{A priori estimates for the Galerkin solution.}\label{sec:apriori}
We begin by stating the main result of this section. 
\begin{theorem}
Suppose that $\bff_f\in H^1(0,T;\bfL^2(\Omega_f))$, $\bff_s\in H^1(0,T;\bfL^2(\Omega_p))$, $\bff_p\in H^1(0,T;L^2(\Omega_p))$ and $P_{in}\in H^1(0,T;H^{1/2}_{00}(\Gamma_f^{in}))$.  In addition, assume that the small data condition \eqref{smalldatacond} holds.
Then, problem {\bf (GF)} has a unique solution $(\bfu_m, \bfeta_m, p_m)$in the interval $[0,T]$. Furthermore, it satisfies the following bounds:
\begin{align}\label{mainbound1}
\dfrac{\rho_f}2\|\bfu_m\|_{\Omega_f}^2
+\frac{\rho_s}2\|\dot{\bfeta}_m\|_{\Omega_p}^2+\mu_s\|\bfD(\bfeta_m)\|^2_{\Omega_p}+\frac{\lambda_s}2\|\nabla\cdot \bfeta_m\|^2_{\Omega_p}+\frac{s_0}2\|p_m\|_{\Omega_p}^2\nonumber\\
+\dfrac12\|\bfK^{1/2}\nabla p_m\|^2_{L^2(0,T;L^2(\Omega_p))}
+\mu_f\|\bfD(\bfu_m)\|^2_{L^2(0,T;L^2(\Omega_f))}\nonumber\\
\leq \left(1+\dfrac{T}{\rho_s}e^{T/\rho_s}\right)\|\calC_1\|^2_{L^2(0,T)}
\end{align}
for all $t\in [0,T]$,
\begin{align}
\|\bfD(\bfu_m)\|_{\Omega_f}&<\dfrac{\mu_f}{3\rho_fS_f^2K_f^3},\label{Dumbound}
\end{align}
and 
\begin{align}
\dfrac{\rho_f}2\|\dot{\bfu}_m\|^2_{\Omega_f}
&+\dfrac{\rho_s}2\|\ddot{\bfeta}_m\|^2_{\Omega_p}+\mu_s\|\bfD(\dot{\bfeta}_m)\|^2_{\Omega_p}+\dfrac{\lambda_s}2\|\nabla\cdot \dot{\bfeta}_m\|^2_{\Omega_p}
+\dfrac{s_0}2\|\dot{p}_m\|^2_{\Omega_p}\nonumber\\
&\hspace{2cm}+\mu_f\|\bfD(\dot{\bfu}_m)\|^2_{L^2(0,T;L^2(\Omega_f))}+\dfrac12\|\bfK^{1/2}\nabla \dot{p}_m\|^2_{L^2(0,T;L^2(\Omega_p))}\nonumber
\\
&\leq (1+\dfrac{T}{\rho_s}e^{2T/\rho_s^2})\|\calC_2\|^2_{L^2(0,T)}+\dfrac{Te^{2T/\rho_s^2}}2\calC_3.\label{mainbound2}
\end{align}
Here $\calC_1, \calC_2$ and $\calC_3$ are defined in \eqref{Cs}.
\end{theorem}
\begin{remark} \label{galerkinbound} 
These bounds imply that $\{\bfu_m\}$ is bounded in $H^1(0,T;\bfH^1(\Omega_f))$, $\{\bfeta_m\}$ is bounded in $H^1(0,T;\bfH^1(\Omega_p))$, $\ddot{\bfeta}_m$ is bounded in $L^{\infty}(0,T;\bfL^2(\Omega_p))$ and
$\{p_m\}$ is bounded in $H^1(0,T;H^1(\Omega_p))$. 
\end{remark}
In the next few sections, we verify the bounds \eqref{mainbound1}, \eqref{Dumbound} and \eqref{mainbound2} in the interval $[0,T_m]$ which will then imply the global existence of the maximal solution $(\bfu_m,\bfeta_m,p_m)$ in the interval $[0,T]$ as stated in the theorem.
\subsubsection{Proof of \eqref{mainbound1}.}
We let $\bfv=\bfu_m, \bfxi=\dot{\bfeta}_m$, $r=p_m$ in the Galerkin formulation {(\bf GF)}. Then the Cauch-Schwarz inequality, inequalities \eqref{trace1}, \eqref{poincare1}, \eqref{poincare3}, \eqref{sobolev}, \eqref{korn} and assumption \eqref{KSPD} on $\bfK$ imply (after neglecting the $\beta$ term which is nonnegative) 
\begin{align}
&\rho_f(\dot{\bfu}_m,\bfu_m)_{\Omega_f}+2\mu_f\|\bfD(\bfu_m)\|^2_{\Omega_f}
+\rho_s(\ddot{\bfeta}_m, \dot{\bfeta}_m)_{\Omega_p}+2\mu_s(\bfD(\bfeta_m), \bfD(\dot{\bfeta}_m))_{\Omega_p}\nonumber\\
&\hspace{3cm}+\lambda_s(\nabla\cdot \bfeta_m, \nabla\cdot \dot{\bfeta}_m)_{\Omega_p}+s_0(\dot{p}_m,p_m)_{\Omega_p}+\|\bfK^{1/2}\nabla p_m\|^2_{\Omega_p}\nonumber\\
%&\hspace{3cm}+\Sigma_{l=1}^{d-1}\beta\|(\bfu_m-\dot{\bfeta}_m)\cdot \bft_{\Gamma}^l\|^2_{\Gamma_I}\nonumber\\
%&\leq -\rho_f(\bfu_m\cdot\nabla \bfu_m,\bfu_m)_{\Omega_f}+\langle P_{in}(t)\bfn_{\Gamma},\bfu_m\rangle_{\Gamma_f^{in}}+(\bff_f,\bfu_m)_{\Omega_f}+(\bff_s,\dot{\bfeta}_m)_{\Omega_p}+(f_p,p_m)_{\Omega_p}\nonumber
&\leq \rho_fS_f^2K_f^3\|\bfD(\bfu_m)\|_{\Omega_f}^3+T_2K_f\|P_{in}(t)\|_{\Gamma_{f}^{in}}\|\bfD(\bfu_m)\|_{\Omega_f}\label{dum}\\
&\hspace{0.7cm}+P_1K_f\|\bff_f\|_{\Omega_f}\|\bfD(\bfu_m)\|_{\Omega_f}+\|\bff_s\|_{\Omega_p}\|\dot{\bfeta}_m\|_{\Omega_p}+\dfrac{P_3}{K_{\min}^{1/2}}\|f_p\|_{\Omega_p}\|\bfK^{1/2}\nabla p_m\|_{\Omega_p}\nonumber.
\end{align}
Here the only problematic terms on the right hand side are $\|\bfD(\bfu_m)\|_{\Omega_f}^3$ and $\|\dot{\bfeta}_m\|_{\Omega_p}$. The rest can easily be hidden in the left hand side. Observe that since $\bfu_m(0)=\bfzero$ and $\bfu_m$ is continuous, there exists a time $\overline{T}_m$, $0<\overline{T}_m\leq T_m$ such that 
\begin{equation}\label{smallness}
\|\bfD(\bfu_m)\|_{\Omega_f}<\dfrac{\mu_f}{3\rho_fS_f^2K_f^3}\quad \forall t\in [0,\overline{T}_m].
\end{equation} 
In fact, this condition holds true on $[0,T_m]$. For the sake of presentation, we postpone this proof to Section \ref{smallnessproof}. Using this condition together with Young's inequality with $\epsilon>0$, we obtain
\begin{align*}
&\rho_f(\dot{\bfu}_m,\bfu_m)_{\Omega_f}+2\mu_f\|\bfD(\bfu_m)\|^2_{\Omega_f}
+\rho_s(\ddot{\bfeta}_m, \dot{\bfeta}_m)_{\Omega_p}+2\mu_s(\bfD(\bfeta_m), \bfD(\dot{\bfeta}_m))_{\Omega_p}\nonumber\\
&+\lambda_s(\nabla\cdot \bfeta_m, \nabla\cdot \dot{\bfeta}_m)_{\Omega_p}+s_0(\dot{p}_m,p_m)_{\Omega_p}+\|\bfK^{1/2}\nabla p_m\|^2_{\Omega_p}%+\Sigma_{l=1}^{d-1}\beta\|(\bfu_m-\dot{\bfeta}_m)\cdot \bft_{\Gamma}^l\|^2_{\Gamma_I}
\nonumber\\&\leq \dfrac{\mu_f}3\|\bfD(\bfu_m)\|_{\Omega_f}^2+\dfrac{\epsilon T_2^2K_f^2}2\|P_{in}(t)\|^2_{\Gamma_{f}^{in}}+\dfrac{1}{2\epsilon}\|\bfD(\bfu_m)\|^2_{\Omega_f}+\dfrac{\epsilon P_1^2K_f^2}2\|\bff_f\|^2_{\Omega_f}\\
&+\dfrac{1}{2\epsilon}\|\bfD(\bfu_m)\|^2_{\Omega_f}+\dfrac12\|\bff_s\|^2_{\Omega_p}+\dfrac12\|\dot{\bfeta}_m\|^2_{\Omega_p}+\dfrac{P_3^2}{2K_{\min}}\|f_p\|_{\Omega_p}^2+\dfrac12\|\bfK^{1/2}\nabla p_m\|_{\Omega_p}^2
\end{align*}
for all $t\in [0,T_m]$.
Picking $\epsilon=\dfrac{3}{2\mu_f}$, we get
\begin{align}
&\rho_f(\dot{\bfu}_m,\bfu_m)_{\Omega_f}+2\mu_f\|\bfD(\bfu_m)\|^2_{\Omega_f}
+\rho_s(\ddot{\bfeta}_m, \dot{\bfeta}_m)_{\Omega_p}+2\mu_s(\bfD(\bfeta_m), \bfD(\dot{\bfeta}_m))_{\Omega_p}\nonumber\\
&+\lambda_s(\nabla\cdot \bfeta_m, \nabla\cdot \dot{\bfeta}_m)_{\Omega_p}+s_0(\dot{p}_m,p_m)_{\Omega_p}+\|\bfK^{1/2}\nabla p_m\|^2_{\Omega_p}%+\Sigma_{l=1}^{d-1}\beta\|(\bfu_m-\dot{\bfeta}_m)\cdot \bft_{\Gamma}^l\|^2_{\Gamma_I}
\nonumber
\\&\leq \mu_f\|\bfD(\bfu_m)\|_{\Omega_f}^2+\dfrac12\|\bfK^{1/2}\nabla p_m\|_{\Omega_p}^2+\dfrac{3 T_2^2K_f^2}{4\mu_f}\|P_{in}(t)\|^2_{\Gamma_{f}^{in}}+\dfrac{ 3P_1^2K_f^2}{4\mu_f}\|\bff_f\|^2_{\Omega_f}\nonumber\\
&\hspace{0.5cm}+\dfrac{P_3^2}{2K_{\min}}\|f_p\|_{\Omega_p}^2+\dfrac12\|\bff_s\|^2_{\Omega_p}+\dfrac12\|\dot{\bfeta}_m\|^2_{\Omega_p}=\calC_1^2(t)+\dfrac12\|\dot{\bfeta}_m\|^2_{\Omega_p}\label{energy1}
\end{align}
where $\calC_1$ is defined in \eqref{Cs}.
Integrating with respect to $t$, due to \eqref{GalerkinIC}, we get:
\begin{align}
\frac{\rho_f}2\|\bfu_m\|_{\Omega_f}^2
+\frac{\rho_s}2\|\dot{\bfeta}_m\|_{\Omega_p}^2+\mu_s\|\bfD(\bfeta_m)\|^2_{\Omega_p}+\frac{\lambda_s}2\|\nabla\cdot \bfeta_m\|^2_{\Omega_p}+\frac{s_0}2\|p_m\|_{\Omega_p}^2\nonumber\\
+\mu_f\int_0^t\|\bfD(\bfu_m)\|^2_{\Omega_f}dt+\dfrac12\int_0^t\|\bfK^{1/2}\nabla p_m\|^2_{\Omega_p}dt\nonumber\\\leq \|\calC_1\|^2_{L^2(0,T)}+\dfrac12\int_0^t\|\dot{\bfeta}_m(s)\|^2_{\Omega_p}ds.\label{energy2}
\end{align}
So 
\begin{align*}
\|\dot{\bfeta}_m(t)\|_{\Omega_p}^2
\leq \dfrac2{\rho_s}\|\calC_1\|^2_{L^2(0,T)}+\dfrac1{\rho_s}\int_0^t\|\dot{\bfeta}_m(s)\|^2_{\Omega_p}ds.
\end{align*}
Therefore, since $\bfeta_m\in \mathcal{C}^2(0,T;\bfX_p^m)$, $\|\dot{\bfeta}_m(t)\|_{\Omega_p}^2\in \calC^1(0,T)$ and applying Gronwall's inequality \eqref{Gronwall} with $\zeta(t)=\|\dot{\bfeta}_m(t)\|_{\Omega_p}^2$, $C=\dfrac1{\rho_s}$ and $B=\dfrac2{\rho_s}\|\calC_1\|^2_{L^2(0,T)}$ yields
\begin{align}
\|\dot{\bfeta}_m(t)\|_{\Omega_p}^2\leq \dfrac2{\rho_s}e^{T/\rho_s}\|\calC_1\|_{L^2(0,T)}^2.\label{etapbd}
\end{align}
Plugging this in \eqref{energy1}, we have
\begin{align}
&\rho_f(\dot{\bfu}_m,\bfu_m)_{\Omega_f}+\mu_f\|\bfD(\bfu_m)\|^2_{\Omega_f}
+\rho_s(\ddot{\bfeta}_m, \dot{\bfeta}_m)_{\Omega_p}+2\mu_s(\bfD(\bfeta_m), \bfD(\dot{\bfeta}_m))_{\Omega_p}\nonumber\\
&+\lambda_s(\nabla\cdot \bfeta_m, \nabla\cdot \dot{\bfeta}_m)_{\Omega_p}+s_0(\dot{p}_m,p_m)_{\Omega_p}+\dfrac12\|\bfK^{1/2}\nabla p_m\|^2_{\Omega_p}\nonumber\\
&\hspace{6cm}\leq \calC_1^2(t)+\dfrac1{\rho_s}e^{T/\rho_s}\|\calC_1\|^2_{L^2(0,T)}\label{um1}
\end{align}
or
\begin{align}
\dfrac{\rho_f}2\dfrac{d}{dt}\|\bfu_m\|_{\Omega_f}^2+\mu_f\|\bfD(\bfu_m)\|^2_{\Omega_f}
+\frac{\rho_s}2\frac{d}{dt}\|\dot{\bfeta}_m\|_{\Omega_p}^2+\mu_s\frac{d}{dt}\|\bfD(\bfeta_m)\|^2_{\Omega_p}\nonumber\\
+\frac{\lambda_s}2\frac{d}{dt}\|\nabla\cdot \bfeta_m\|^2_{\Omega_p}+\frac{s_0}2\frac{d}{dt}\|p_m\|_{\Omega_p}^2+\dfrac12\|\bfK^{1/2}\nabla p_m\|^2_{\Omega_p}\nonumber\\
\leq \calC_1^2(t)+\dfrac1{\rho_s}e^{T/\rho_s}\|\calC_1\|^2_{L^2(0,T)}\label{um}
\end{align}
for all $t\in [0,T_m]$.
After integrating with respect to $t$ and using \eqref{GalerkinIC} implies the bound \eqref{mainbound1}.
\subsubsection{Proof of \eqref{Dumbound} and \eqref{mainbound2}.}\label{smallnessproof}
\begin{proof}
Recalling \eqref{smallness}, assume for a contradiction that there exists $T^{*}\in (0,T_m]$ such that 
\begin{equation}
\forall t\in [0,T^{*}), \quad \|\bfD(\bfu_m)(t)\|_{\Omega_f}<\dfrac{\mu_f}{3\rho_fS_f^2K_f^3},\quad \|\bfD(\bfu_m)(T^{*})\|_{\Omega_f}=\dfrac{\mu_f}{3\rho_fS_f^2K_f^3}.\label{assumption}
\end{equation}                                                                                                                                  
Similar arguments leading to \eqref{um1} and using Young's inequality yield 
 \begin{align}
2\mu_f\|&\bfD(\bfu_m)\|^2_{\Omega_f}\nonumber
\\\leq & \dfrac{\rho_f}2\|\dot{\bfu}_m\|_{\Omega_f}^2+\dfrac{\rho_s}2\|\dot{\bfeta}_m\|^2_{\Omega_p}+\mu_s\|\bfD(\bfeta_m)\|^2_{\Omega_p}+
\dfrac{\lambda_s}2\|\nabla\cdot \bfeta_m\|^2_{\Omega_p}+\dfrac{s_0}2\|p_m\|_{\Omega_p}^2\nonumber\\
&+\dfrac{\rho_f}2\|\bfu_m\|_{\Omega_f}^2+\dfrac{\rho_s}2\|\ddot{\bfeta}_m\|_{\Omega_p}^2+\mu_s\|\bfD(\dot{\bfeta}_m)\|^2_{\Omega_p}+\dfrac{\lambda_s}2\|\nabla\cdot \dot{\bfeta}_m\|^2_{\Omega_p}+\dfrac{s_0}2\|\dot{p}_m\|^2_{\Omega_p}\nonumber\\
&+\calC_1^2(t)+\dfrac1{\rho_s} e^{T/\rho_s}\|\calC_1(t)\|_{L^2(0,T)}\label{small1}
\end{align}
 for all $t\in [0,T^{*}]$.
To bound the first five terms we differentiate {\bf (GF)} with respect to time. The specifics of this technique can be found in detail in \cite{Lions1961}.
\begin{align*}
&(\rho_f\ddot{\bfu}_m,\bfv)_{\Omega_f}+(2\mu_f\bfD(\dot{\bfu}_m),\bfD(\bfv))_{\Omega_f}+\rho_f(\dot{\bfu}_m\cdot\nabla \bfu_m,\bfv)_{\Omega_f}+\rho_f(\bfu_m\cdot\nabla \dot{\bfu}_m,\bfv)_{\Omega_f}
\\&+(\rho_s\dddot{\bfeta}_m,\bfxi)_{\Omega_p}+(2\mu_s\bfD(\dot{\bfeta}_m),\bfD(\bfxi))_{\Omega_p}+(\lambda_s\nabla\cdot \dot{\bfeta}_m,\nabla\cdot\bfxi)_{\Omega_p}-(\alpha \dot{p}_m,\nabla\cdot \bfxi)_{\Omega_p}\\&+(s_0\ddot{p}_m+\alpha\nabla\cdot \ddot{\bfeta}_m,r)_{\Omega_p}+(\bfK\nabla \dot{p}_m,\nabla r)_{\Omega_p}
+\langle \dot{p}_m\bfn_\Gamma,\bfv-\bfxi\rangle_{\Gamma_I}\\&+\Sigma_{l=1}^{d-1}\langle\beta(\dot{\bfu}_m-\ddot{\bfeta}_m)\cdot \bft_{\Gamma}^l,(\bfv-\bfxi)\cdot \bft_{\Gamma}^l\rangle_{\Gamma_I}+\langle(\ddot{\bfeta}_m-\dot{\bfu}_m)\cdot\bfn_{\Gamma},r\rangle_{\Gamma_I}\\
&\hspace{1.8cm}=-\langle \dot{P_{in}}\bfn_{\Gamma},\bfv\rangle_{\Gamma_f^{in}}+(\dot{\bff}_f,\bfv)_{\Omega_f}+(\dot{\bff}_s,\bfxi)_{\Omega_p}+(\dot{f}_p,r)_{\Omega_p}.
\end{align*}
We let $\bfv=\dot{\bfu}_m$, $r=\dot{p}_m$ and $\bfxi=\ddot{\bfeta}_m$ in the above, use the Cauch-Schwarz inequality, assumption \eqref{assumption}, inequalities \eqref{trace1}, \eqref{poincare1}, \eqref{poincare3}, \eqref{sobolev}, \eqref{korn} and assumption \eqref{KSPD} on $\bfK$ to get (neglecting the nonnegative $\beta$ term)
\begin{align*}
&\dfrac{\rho_f}2\dfrac{d}{dt}\|\dot{\bfu}_m\|^2_{\Omega_f}+2\mu_f\|\bfD(\dot{\bfu}_m)\|^2_{\Omega_f}
+\dfrac{\rho_s}2\dfrac{d}{dt}\|\ddot{\bfeta}_m\|^2_{\Omega_p}+\mu_s\dfrac{d}{dt}\|\bfD(\dot{\bfeta}_m)\|^2_{\Omega_p}\\
&+\dfrac{\lambda_s}2\dfrac{d}{dt}\|\nabla\cdot \dot{\bfeta}_m\|^2_{\Omega_p}+\dfrac{s_0}2\dfrac{d}{dt}\|\dot{p}_m\|^2_{\Omega_p}+\|\bfK^{1/2}\nabla \dot{p}_m\|^2_{\Omega_p}
%+\Sigma_{l=1}^{d-1}\beta\|(\dot{\bfu}_m-\ddot{\bfeta}_m)\cdot \bft_{\Gamma}^l\|^2_{\Gamma_I}
\\
&\leq -\rho_f(\dot{\bfu}_m\cdot\nabla \bfu_m,\dot{\bfu}_m)_{\Omega_f}-\rho_f(\bfu_m\cdot\nabla \dot{\bfu}_m,\dot{\bfu}_m)_{\Omega_f}\\&\hspace{3cm}-\langle \dot{P_{in}}\bfn_{\Gamma},\dot{\bfu}_m\rangle_{\Gamma_f^{in}}+(\dot{\bff}_f,\dot{\bfu}_m)_{\Omega_f}+(\dot{\bff}_s,\ddot{\bfeta}_m)_{\Omega_p}+(\dot{f}_p,\dot{p}_m)_{\Omega_p}
\end{align*}
\begin{align*}
\leq &2\rho_fS_f^2K_f^3\|\bfD(\dot{\bfu}_m)\|^2_{\Omega_f}\|\bfD(\bfu_m)\|_{\Omega_f}+T_2K_f\|\dot{P}_{in}\|_{\Gamma_f^{in}}\|\bfD(\dot{\bfu}_m)\|_{\Omega_f}\\
&+P_1K_f\|\dot{\bff_f}\|_{\Omega_f}\|\bfD(\dot{\bfu}_m)\|_{\Omega_f}+\|\dot{\bff}_s\|_{\Omega_p}\|\ddot{\bfeta}_m\|_{\Omega_p}+\dfrac{P_3}{K_{\min}^{1/2}}\|\dot{f_p}\|_{\Omega_p}\|\bfK^{1/2}\nabla \dot{p}_m\|_{\Omega_p}\\
\leq &\dfrac{2\mu_f}3\|\bfD(\dot{\bfu}_m)\|^2_{\Omega_f}+\dfrac1{2\epsilon}\|\bfD(\dot{\bfu}_m)\|^2_{\Omega_f}+\dfrac{\epsilon T_2^2K_f^2}2\|\dot{P}_{in}\|_{\Gamma_f^{in}}^2+\dfrac1{2\epsilon}\|\bfD(\dot{\bfu}_m)\|^2_{\Omega_f}\\&+\dfrac{\epsilon P_1^2K_f^2}2\|\dot{\bff_f}\|^2_{\Omega_f}+\|\dot{\bff}_s\|_{\Omega_p}\|\ddot{\bfeta}_m\|_{\Omega_p}+\dfrac12\|\bfK^{1/2}\nabla \dot{p}_m\|_{\Omega_p}^2+\dfrac{P_3^2}{2K_{\min}}\|\dot{f_p}\|_{\Omega_p}^2.
\end{align*}
Picking $\epsilon=\dfrac3{\mu_f}$, we obtain
\begin{align}
&\dfrac{\rho_f}2\dfrac{d}{dt}\|\dot{\bfu}_m\|^2_{\Omega_f}+\mu_f\|\bfD(\dot{\bfu}_m)\|^2_{\Omega_f}
+\dfrac{\rho_s}2\dfrac{d}{dt}\|\ddot{\bfeta}_m\|^2_{\Omega_p}+\mu_s\dfrac{d}{dt}\|\bfD(\dot{\bfeta}_m)\|^2_{\Omega_p}\nonumber\\
&\hspace{3.8cm}+\dfrac{\lambda_s}2\dfrac{d}{dt}\|\nabla\cdot \dot{\bfeta}_m\|^2_{\Omega_p}+\dfrac{s_0}2\dfrac{d}{dt}\|\dot{p}_m\|^2_{\Omega_p}+\dfrac12\|\bfK^{1/2}\nabla \dot{p}_m\|^2_{\Omega_p}\nonumber
\\
&\leq \dfrac{3 T_2^2K_f^2}{2\mu_f}\|\dot{P}_{in}\|_{\Gamma_f^{in}}^2+\dfrac{3 P_1^2K_f^2}{2\mu_f}\|\dot{\bff_f}\|^2_{\Omega_f}+\dfrac12\|\dot{\bff}_s\|^2_{\Omega_p}+\dfrac12\|\ddot{\bfeta}_m\|^2_{\Omega_p}+\dfrac{P_3^2}{2K_{\min}}\|\dot{f_p}\|_{\Omega_p}^2\nonumber\\
&\hspace{6cm}=\calC_2^2(t)+\dfrac12\|\ddot{\bfeta}_m\|^2_{\Omega_p},\label{dotum}
\end{align}
where $\calC_2\in L^2(0,T)$ is defined in \eqref{Cs}.
Now we need a bound for $\ddot{\bfeta}_m$. Since by assumption $\ddot{\bfeta}$ is continuous, we can use Gronwall's inequality \eqref{Gronwall} again. 
Integrating from 0 to $t$ for any $t\in [0,T^{*}]$ and recalling \eqref{GalerkinIC} we obtain
\begin{align}
&\dfrac{\rho_f}2\|\dot{\bfu}_m\|^2_{\Omega_f}
+\dfrac{\rho_s}2\|\ddot{\bfeta}_m\|^2_{\Omega_p}+\mu_s\|\bfD(\dot{\bfeta}_m)\|^2_{\Omega_p}
+\dfrac{\lambda_s}2\|\nabla\cdot \dot{\bfeta}_m\|^2_{\Omega_p}+\dfrac{s_0}2\|\dot{p}_m\|^2_{\Omega_p}\nonumber\\
%&\hspace{4cm}+\int_0^t\mu_f\|\bfD(\dot{\bfu}_m)\|^2_{\Omega_f}dt+\int_0^t \dfrac12\|\bfK^{1/2}\nabla \dot{p}_m\|^2_{\Omega_p}dt\nonumber
%\\
&\leq\dfrac{\rho_f}2\|\dot{\bfu}_m(0)\|^2_{\Omega_f}
+\dfrac{\rho_s}2\|\ddot{\bfeta}_m(0)\|^2_{\Omega_p}+\dfrac{s_0}2\|\dot{p}_m(0)\|^2_{\Omega_p}\nonumber\\
&\hspace{4cm}+\|\calC_2\|^2_{L^2(0,T)}+\dfrac1{\rho_s}\int_0^t\|\ddot{\bfeta}_m(s)\|^2_{\Omega_p}ds\label{ddoteta}.
\end{align}
Now we need bounds for $\|\dot{\bfu}_m(0)\|^2_{\Omega_f}$, $\|\ddot{\bfeta}_m(0)\|^2_{\Omega_p}$ and $\|\dot{p}_m(0)\|^2_{\Omega_p}$. For that purpose we let $\bfv=\bfzero$, $r=0$ and $\bfxi=\ddot{\bfeta}_m(0)$ in {\bf (GF)} at $t=0$. Due to \eqref{GalerkinIC} this yields:
\begin{align*}
\rho_s\|\ddot{\bfeta}_m(0)\|_{\Omega_p}^2=(\bff_s(0), \ddot{\bfeta}_m(0))_{\Omega_p}\leq \|\bff_s(0)\|_{\Omega_p}\|\ddot{\bfeta}_m(0)\|_{\Omega_p}.
\end{align*}
Therefore, 
\begin{equation}
\label{ddoteta0}
\|\ddot{\bfeta}_m(0)\|_{\Omega_p}\leq \dfrac1{\rho_s}\|\bff_s(0)\|_{\Omega_p}.
\end{equation}
Now let $\bfv=\bfeta=\bfzero$ and $r=\dot{p}_m(0)$ and $t=0$ in {(\bf GF)}. Then by \eqref{GalerkinIC} we have 
\begin{equation}
\label{dotp0}
\|\dot{p}_m(0)\|_{\Omega_p}\leq \dfrac1{s_0}\|f_p(0)\|_{\Omega_p}.
\end{equation}
Furthermore if we let $\bfeta=\bfzero$, $r=0$, $\bfv=\dot{\bfu}_m(0)$ and $t=0$ in {(\bf GF)}, we get 
\[
\rho_f\|\dot{\bfu}_m(0)\|^2_{\Omega_f}=-\langle P_{in}(0)\bfn_{\Gamma_f^{in}}, \dot{\bfu}_m(0)\rangle_{\Gamma_f^{in}}+(\bff_f(0), \dot{\bfu}_m(0))_{\Omega_f}.
\]
We need a good bound for $\langle P_{in}\bfn_{\Gamma}, \dot{\bfu}_m(0)\rangle_{\Gamma_f^{in}}$ because otherwise $\|\dot{\bfu}_m(0)\|_{\Gamma_f^{in}}$ cannot be hidden in the left hand side. 
Assuming $P_{in}\in L^2(0,T;H_{00}^{1/2}(\Gamma_f^{in}))$, its extension  $\tilde{P}_{in}$ by zero  to $\partial \Omega_f$ is in $L^2(0,T;H^{1/2}(\partial \Omega_f)).$ Then we can use the continuous lifting operator $j:L^2(0,T;H^{1/2}(\partial \Omega_f))\rightarrow L^2(0,T;H^1( \Omega_f))$ with continuity constant $C_j$. Since $\nabla\cdot\dot{\bfu}_m=0$
\begin{align*}\langle P_{in}(0)\bfn_{\Gamma_f^{in}}&, \dot{\bfu}_m(0)\rangle_{\Gamma_f^{in}}=\langle j(\tilde{P}_{in}(0))\bfn_{\partial\Omega_f}, \dot{\bfu}_m(0)\rangle_{\partial\Omega_f}=(\nabla (j(\tilde{P}_{in}(0))), \dot{\bfu}_m(0) )_{\Omega_f}\\
&\leq \|\nabla (j(\tilde{P}_{in}(0))\|_{\Omega_f}\|\dot{\bfu}_m(0) \|_{\Omega_f}\leq  \|j(\tilde{P}_{in}(0))\|_{H^1(\Omega_f)}\|\dot{\bfu}_m(0) \|_{\Omega_f}\\
&\leq C_j\|\tilde{P}_{in}(0)\|_{H^{1/2}(\partial\Omega_f)}\|\dot{\bfu}_m(0) \|_{\Omega_f}=C_j\|P_{in}(0)\|_{H_{00}^{1/2}(\Gamma_f^{in})}\|\dot{\bfu}_m(0) \|_{\Omega_f}.
\end{align*}
Then 
\begin{equation}\label{umprime0}
\|\dot{\bfu}_m(0)\|_{\Omega_f}\leq \dfrac{C_j}{\rho_f}\|P_{in}(0)\|_{H_{00}^{1/2}(\Gamma_f^{in})}+\dfrac{1}{\rho_f}\|\bff_f(0)\|_{\Omega_f}.
\end{equation}
 Plugging these in \eqref{ddoteta}, 
\begin{align*}
&\dfrac{\rho_f}2\|\dot{\bfu}_m\|^2_{\Omega_f}
+\dfrac{\rho_s}2\|\ddot{\bfeta}_m\|^2_{\Omega_p}+\mu_s\|\bfD(\dot{\bfeta}_m)\|^2_{\Omega_p}
+\dfrac{\lambda_s}2\|\nabla\cdot \dot{\bfeta}_m\|^2_{\Omega_p}+\dfrac{s_0}2\|\dot{p}_m\|^2_{\Omega_p}\nonumber
%&\hspace{4cm}+\int_0^t\mu_f\|\bfD(\dot{\bfu}_m)\|^2_{\Omega_f}dt+\int_0^t \dfrac12\|\bfK^{1/2}\nabla \dot{p}_m\|^2_{\Omega_p}dt\nonumber
\\&\leq \dfrac{C_f^2}{\rho_f}\|P_{in}(0)\|^2_{H_{00}^{1/2}(\Gamma_f^{in})}+\dfrac1{\rho_f}\|\bff_f(0)\|_{\Omega_f}^2
+\dfrac1{2\rho_s}\|\bff_s(0)\|^2_{\Omega_p}+\dfrac1{2s_0}\|f_p(0)\|^2_{\Omega_p}\\&\hspace{4cm}+\|\calC_2\|^2_{L^2(0,T)}
+\dfrac1{\rho_s}\int_0^t\|\ddot{\bfeta}_m(s)\|^2_{\Omega_p}ds\\
&= \calC_3+\|\calC_2\|^2_{L^2(0,T)}
+\dfrac1{\rho_s}\int_0^t\|\ddot{\bfeta}_m(s)\|^2_{\Omega_p}ds
\end{align*}
where 
$\calC_3$ is defined in \eqref{Cs}.
If we neglect all terms other than the second one on the left hand side of the above equation we get
\begin{align*}
\|\ddot{\bfeta}_m\|^2_{\Omega_p}
\leq \dfrac2{\rho_s}\calC_3+\dfrac2{\rho_s}\|\calC_2\|^2_{L^2(0,T)}
+\dfrac2{\rho_s^2}\int_0^t\|\ddot{\bfeta}_m(s)\|^2_{\Omega_p}ds.
\end{align*}
 Using Gronwall's inequality \eqref{Gronwall} with
$B= \dfrac2{\rho_s}\calC_3+\dfrac2{\rho_s}\|\calC_2\|^2_{L^2(0,T)}$ , $C=\dfrac2{\rho_s^2}$ and $\zeta(t)=\|\ddot{\bfeta}_m(t)\|^2_{\Omega_p}$,
\begin{align*}
\|\ddot{\bfeta}_m(t)\|^2_{\Omega_p}\leq& e^{2T/\rho_s^2}\Big(\dfrac2{\rho_s}\|\calC_2\|^2_{L^2(0,T)}+\dfrac2{\rho_s}\calC_3\Big).
\end{align*}
Therefore, putting this in \eqref{dotum} we have
\begin{align}
&\dfrac{\rho_f}2\dfrac{d}{dt}\|\dot{\bfu}_m\|^2_{\Omega_f}+\mu_f\|\bfD(\dot{\bfu}_m)\|^2_{\Omega_f}
+\dfrac{\rho_s}2\dfrac{d}{dt}\|\ddot{\bfeta}_m\|^2_{\Omega_p}+\mu_s\dfrac{d}{dt}\|\bfD(\dot{\bfeta}_m)\|^2_{\Omega_p}\nonumber\\
&\hspace{4cm}+\dfrac{\lambda_s}2\dfrac{d}{dt}\|\nabla\cdot \dot{\bfeta}_m\|^2_{\Omega_p}
+\dfrac{s_0}2\dfrac{d}{dt}\|\dot{p}_m\|^2_{\Omega_p}+\dfrac12\|\bfK^{1/2}\nabla \dot{p}_m\|^2_{\Omega_p}
\nonumber\\
&\hspace{3cm}\leq \calC_2^2(t)+\dfrac{e^{2T/\rho_s^2}}{\rho_s}\Big(\|\calC_2\|^2_{L^2(0,T)}+\calC_3\Big).\label{ineq1}
\end{align}
Integrating from 0 to $t$ in $[0,T^{*}]$ and recalling the bounds \eqref{ddoteta0}, \eqref{dotp0} and \eqref{umprime0} we obtain
\begin{align}
\dfrac{\rho_f}2\|\dot{\bfu}_m\|^2_{\Omega_f}
&+\dfrac{\rho_s}2\|\ddot{\bfeta}_m\|^2_{\Omega_p}+\mu_s\|\bfD(\dot{\bfeta}_m)\|^2_{\Omega_p}+\dfrac{\lambda_s}2\|\nabla\cdot \dot{\bfeta}_m\|^2_{\Omega_p}
+\dfrac{s_0}2\|\dot{p}_m\|^2_{\Omega_p}\nonumber\\
&\hspace{2cm}+\int_0^t\mu_f\|\bfD(\dot{\bfu}_m)\|^2_{\Omega_f}dt+\int_0^t\dfrac12\|\bfK^{1/2}\nabla \dot{p}_m\|^2_{\Omega_p}dt\nonumber
\\
&\leq \|\calC_2\|^2_{L^2(0,T)}+\dfrac{T}{\rho_s}e^{2T/\rho_s^2}\|\calC_2\|^2_{L^2(0,T)}+\dfrac{Te^{2T/\rho_s^2}}2\calC_3\nonumber\\
&=(1+\dfrac{T}{\rho_s}e^{2T/\rho_s^2})\|\calC_2\|^2_{L^2(0,T)}+\dfrac{Te^{2T/\rho_s^2}}2\calC_3,\label{ineq:bd2}
\end{align}
which implies
\eqref{mainbound2}. Neglecting all terms on the left hand side other than the first we find

\begin{equation} \label{umprime}
\|\dot{\bfu}_m\|^2_{\Omega_f}
\leq \dfrac2{\rho_f}\left((1+\dfrac{T}{\rho_s}e^{2T/\rho_s^2})\|\calC_2\|^2_{L^2(0,T)}+\dfrac{Te^{2T/\rho_s^2}}2\calC_3\right).
\end{equation}
To find a bound for $\|\bfu_m\|_{\Omega_f}$ in \eqref{small1},
we recall \eqref{dum}, use assumption \eqref{assumption} to get \eqref{um} for all $t\in [0,T^{*}]$. Neglecting the terms other than the ones with $d/dt$ on the left hand side, we integrate \eqref{um} from 0 to $t$ where $t\in[0,T^{*}]$. Since the initial conditions are zero, we have
\begin{align}
\dfrac{\rho_f}2\|\bfu_m\|_{\Omega_f}^2
+\frac{\rho_s}2\|\dot{\bfeta}_m\|_{\Omega_p}^2+\mu_s\|\bfD(\bfeta_m)\|^2_{\Omega_p}
+\frac{\lambda_s}2\|\nabla\cdot \bfeta_m\|^2_{\Omega_p}+\frac{s_0}2\|p_m\|_{\Omega_p}^2\nonumber\\
\leq (1+\dfrac{T}{\rho_s}e^{T/\rho_s})\|\calC_1\|^2_{L^2(0,T)}.\label{ineq:bd1}
\end{align}
Therefore using \eqref{umprime} and \eqref{ineq:bd1} in \eqref{small1} we get
\begin{multline*}\|\bfD(\bfu_m)\|_{\Omega_f}^2\leq \dfrac1{2\mu_f}\Big(\calC_1^2(t)
+\big(1+\dfrac1{\rho_s} e^{T/\rho_s}+\dfrac{T}{\rho_s}e^{T/\rho_s}\big)\|\calC_1\|^2_{L^2(0,T)}\\
+\big(1+\dfrac{T}{\rho_s}e^{2T/\rho_s^2}\big)\|\calC_2\|^2_{L^2(0,T)}+\dfrac{Te^{2T/\rho_s^2}}2\calC_3\Big)
\end{multline*}
for all $t\in [0,T^{*}]$.
So if the data satisfy
\begin{multline*}\dfrac1{2\mu_f}\Big(\|\calC_1\|^2_{L^{\infty}(0,T)}
+\big(1+\dfrac1{\rho_s} e^{T/\rho_s}+\dfrac{T}{\rho_s}e^{T/\rho_s}\big)\|\calC_1\|^2_{L^2(0,T)}\\
+\big(1+\dfrac{T}{\rho_s}e^{2T/\rho_s^2}\big)\|\calC_2\|^2_{L^2(0,T)}+\dfrac{Te^{2T/\rho_s^2}}2\calC_3\Big)<\dfrac{\mu_f^3}{9\rho_f^2S_f^4K_f^6},
\end{multline*} 
then we have
\[
\|\bfD(\bfu_m)\|_{\Omega_f}<\dfrac{\mu_f^2}{3\rho_fS_f^2K_f^3}
\]
for all $t\in [0,T^{*}]$
which contradicts with the assumption.
\end{proof}
\begin{remark}
Note that if $P_{in}(0)=0$,  we immediately get $\|\dot{\bfu}_m(0)\|_{\Omega_f}\leq \dfrac1{\rho_f}\|\bff_f(0)\|_{\Omega_f}$ rather than \eqref{umprime0}. In that case, it is not necessary to assume that $P_{in}\in L^2(0,T;H_{00}^{1/2}(\Gamma_f^{in}))$. 
\end{remark}
\subsection{Passing to the limit.}
Since we established the boundedness of the Galerkin solution $(\bfu_m, \bfeta_m, p_m)$ that we stated in Remark \ref{galerkinbound}, we can pass to the limit in {(\bf GF)} to obtain a solution to {(\bf WF2)}.
Reflexivity of $\bfV_f$, $\bfX_p$ and $Q_p$ and therefore of $H^1(0,T;\bfV_f)$, $H^1(0,T;\bfX_p)$ and $H^1(0,T;Q_p)$ imply that there exists a subsequence of the Galerkin solution, still denoted by $\{(\bfu_m, \bfeta_m, p_m)\}_m$
such that
\begin{alignat*}{3}
\bfu_m&\rightharpoonup &\,\, \bfu_f\quad &\mbox{ in }H^1(0,T;\bfH^1(\Omega_f)), 
%\\
%\dot{\bfu}_m& \rightharpoonup &\,\,{\dot\bfu}_f\quad &\mbox{ weakly in }L^2(0,T;\bfH^1(\Omega_f)),
\\
\bfeta_m&\rightharpoonup&\bfeta\quad &\mbox{ in }H^1(0,T;\bfH^1(\Omega_p)),
\\
\ddot{\bfeta}_m&\rightharpoonup&\ddot{\bfeta}\quad &\mbox{  in } L^2(0,T;\bfL^2(\Omega_p)), %L^2(0,T;\bfX_p^{*}),
\\
p_m&\rightharpoonup&p_p\quad& \mbox{  in } H^1(0,T;H^1(\Omega_p))
\end{alignat*}
where $\rightharpoonup$ stands for weak convergence.
Note that due to the continuity of the trace operator we also have
\begin{alignat*}{3}
\bfu_m&\rightharpoonup &\,\, \bfu_f\quad &\mbox{ in }L^2(0,T;H^{1/2}(\partial\Omega_f)), 
\\
\dot{\bfeta}_m&\rightharpoonup&\dot{\bfeta}\quad &\mbox{ in }L^2(0,T;H^{1/2}(\partial\Omega_p)),\\
p_m&\rightharpoonup&p_p\quad& \mbox{ in } L^2(0,T;H^{1/2}(\partial\Omega_p)).
\end{alignat*}
We can pass to the limit in the linear terms with the convergence results above. However, the nonlinear term needs a stronger convergence result. To obtain such a result we recall that $\bfH^1(\Omega_f)\subset \bfL^4(\Omega_f)\subset \bfL^2(\Omega_f)$ and that $\bfH^1(\Omega_f)$ is compactly embedded in $\bfL^4(\Omega_f)$ due to the Sobolev embedding theorem. Furthermore, Remark \ref{galerkinbound} implies that the subsequence $\{\bfu_m\}_m$ is bounded in $L^4(0,T;\bfH^1(\Omega_f))$ and $\{\frac{\partial \bfu_m}{\partial t}\}_m$ is 
bounded in $L^1(0,T;\bfL^2(\Omega_f))$. Therefore, Theorem \ref{Simon} implies that $\{\bfu_m\}_m$ has a subsequence, still denoted the same, such that 
\[
\bfu_m\rightarrow \bfu_f \mbox{ strongly in }L^4(0,T;\bfL^4(\Omega_f)).
\] 
Fix an integer $k\geq 1$. We multiply {(\bf GF)} by $\psi(t)\in L^2(0,T)$ 
 and integrate in time to obtain
\begin{align}
\int_0^T(\rho_f\bfudot_m&,\psi(t)\bfv)_{\Omega_f}dt+\int_0^T(2\mu_f\bfD(\bfu_m),\psi(t)\bfD(\bfv))_{\Omega_f}dt\nonumber\\
&+\int_0^T\rho_f(\bfu_m\cdot\nabla \bfu_m,\psi(t)\bfv)_{\Omega_f}dt
+\int_0^T(\rho_s\ddot{\bfeta}_m,\psi(t)\bfxi)_{\Omega_p}dt\nonumber\\
&+\int_0^T(2\mu_s\bfD(\bfeta_m),\psi(t)\bfD(\bfxi))_{\Omega_p}dt
+\int_0^T(\lambda_s\nabla\cdot \bfeta_m,\psi(t)\nabla\cdot\bfxi)_{\Omega_p}dt\nonumber\\
&-\int_0^T(\alpha p_m,\psi(t)\nabla\cdot \bfxi)_{\Omega_p}dt+\int_0^T(s_0\dot{p}_m+\alpha\nabla\cdot \bfetadot_m,\psi(t)r)_{\Omega_p}dt\nonumber\\
&+\int_0^T(\bfK\nabla p_m,\psi(t)\nabla r)_{\Omega_p}dt\nonumber
+\int_0^T\langle p_m\bfn_\Gamma,\psi(t)(\bfv-\bfxi)\rangle_{\Gamma_I}dt\\
&+\Sigma_{l=1}^{d-1}\int_0^T\langle\beta(\bfu_m-\bfetadot_m)\cdot \bft_{\Gamma}^l,\psi(t)(\bfv-\bfxi)\cdot \bft_{\Gamma}^l\rangle_{\Gamma_I}dt\nonumber\\
&+\int_0^T\langle(\bfetadot_m-\bfu_m)\cdot\bfn_{\Gamma},\psi(t)r\rangle_{\Gamma_I}dt\label{init1}=-\int_0^T\langle P_{in}(t)\bfn_{\Gamma},\psi(t)\bfv\rangle_{\Gamma_f^{in}}dt\\
&+\int_0^T(\bff_f,\psi(t)\bfv)_{\Omega_f}dt+\int_0^T(\bff_s,\psi(t)\bfxi)_{\Omega_p}dt+\int_0^T(f_p,\psi(t)r)_{\Omega_p}dt,\nonumber
\end{align}
for all $(\bfv, \bfxi, r)\in \bfV_f^k\times \bfX_p^k\times Q_p^k$, $m\geq k$. Letting $m\rightarrow \infty$ we obtain
\begin{align}
\int_0^T(\rho_f\bfudot_f,&\psi(t)\bfv)_{\Omega_f}dt+\int_0^T(2\mu_f\bfD(\bfu_f),\psi(t)\bfD(\bfv))_{\Omega_f}dt\nonumber\\
&+\int_0^T\rho_f(\bfu_f\cdot\nabla \bfu_f,\psi(t)\bfv)_{\Omega_f}dt\nonumber
+\int_0^T(\rho_s\ddot{\bfeta},\psi(t)\bfxi)_{\Omega_p}dt\\
&+\int_0^T(2\mu_s\bfD(\bfeta),\psi(t)\bfD(\bfxi))_{\Omega_p}dt
+\int_0^T(\lambda_s\nabla\cdot \bfeta,\psi(t)\nabla\cdot\bfxi)_{\Omega_p}dt\nonumber\\
&-\int_0^T(\alpha p_p,\psi(t)\nabla\cdot \bfxi)_{\Omega_p}dt+\int_0^T(s_0\dot{p}_p+\alpha\nabla\cdot \bfetadot,\psi(t)r)_{\Omega_p}dt\nonumber\\
&+\int_0^T(\bfK\nabla p_p,\psi(t)\nabla r)_{\Omega_p}dt
+\int_0^T\langle p_p\bfn_\Gamma,\psi(t)(\bfv-\bfxi)\rangle_{\Gamma_I}dt\nonumber\\
&+\Sigma_{l=1}^{d-1}\int_0^T\langle\beta(\bfu_f-\bfetadot)\cdot \bft_{\Gamma}^l,\psi(t)(\bfv-\bfxi)\cdot \bft_{\Gamma}^l\rangle_{\Gamma_I}dt\label{init2}\\
&+\int_0^T\langle(\bfetadot-\bfu_f)\cdot\bfn_{\Gamma},\psi(t)r\rangle_{\Gamma_I}dt=-\int_0^T\langle P_{in}(t)\bfn_{\Gamma},\psi(t)\bfv\rangle_{\Gamma_f^{in}}dt\nonumber\\
&+\int_0^T(\bff_f,\psi(t)\bfv)_{\Omega_f}dt+\int_0^T(\bff_s,\psi(t)\bfxi)_{\Omega_p}dt+\int_0^T(f_p,\psi(t)r)_{\Omega_p}dt,\nonumber
\end{align}
for all $(\bfv, \bfxi, r)\in \bfV_f^k\times \bfX_p^k\times Q_p^k$. Since any element of $\bfV, \bfX_p, Q_p$ can be approximated by elements of $\bfV_f^k, \bfX_p^k, Q_p^k$ and $\psi\in L^2(0,T)$ is arbitrary, this also holds for all $(\bfv, \bfxi, r)\in \bfV_f\times \bfX_p\times Q_p$ a.e. in $(0,T)$. Therefore we recover the equation in {(\bf WF2)}.

Last, we check 
whether the initial conditions are satisfied. 
In \eqref{init2} we let $\psi\in \calC^2([0,T])$ such that $\psi(T)=0$ and $\dot{\psi}(T)=0$ and integrate the first and the eighth terms once and the fourth term twice with respect to time. Then we do the same in \eqref{init1} and also take the limit as $m\rightarrow \infty$ using the convergence properties established in the beginning of the proof and using \eqref{GalerkinIC}.  
Comparing the resulting equations yield:
\begin{multline*}
-(\rho_f\bfu(0),\psi(0)\bfv)_{\Omega_f}-(\rho_s\dot{\bfeta}(0),\psi(0)\bfxi)_{\Omega_p}+(\rho_s\bfeta(0),\dot{\psi}(0)\bfxi)_{\Omega_p}\\
-(s_0 p(0),\psi(0)r)_{\Omega_p}=0,
\end{multline*}
for all $\bfv\in \bfV, \bfxi\in \bfX_p, r\in Q_p$. Since $\psi(0)$ and $\dot{\psi}(0)$ are arbitrary, 
\begin{align*}
-(\rho_f\bfu(0),\bfv)_{\Omega_f}-(\rho_s\dot{\bfeta}(0),\bfxi)_{\Omega_p}+(\rho_s\bfeta(0),\bfxi)_{\Omega_p}
-(s_0 p(0),r)_{\Omega_p}=0,
\end{align*}
for all $\bfv\in \bfV, \bfxi\in \bfX_p, r\in Q_p$.
This yields the initial conditions stated in {(\bf WF2)}.
Finally passing to the limit in \eqref{mainbound1}, \eqref{Dumbound} and \eqref{mainbound2}, we obtain \eqref{ineq:mainresult1}, \eqref{ineq:mainresult2} and \eqref{ineq:mainresult3}.
\subsection{Uniqueness.}
For the Stokes flow, there is no issue of uniqueness. For the Navier-Stokes problem, we can only prove  uniqueness for restricted solution.
\begin{theorem}[Local uniqueness]
Let $(\bfu_f, \bfeta, p_p)$ be a solution of {(\bf WF2)}.  Then if 
\[\|\bfD(\bfu)\|_{\Omega_f}\leq \dfrac{\mu}{S_f^2K_f^3},\]
then $(\bfu_f, \bfeta, p_p)$ is unique.
\end{theorem}
\begin{proof}
Let $(\bfu_1, \bfeta_1, p_1)$ and $(\bfu_2, \bfeta_2, p_2)$ be two solutions of {(\bf WF2)}. Then $\bfw=\bfu_1-\bfu_2$, $\bftheta=\bfeta_1-\bfeta_2$ and $\phi=p_1-p_2$ satisfy:
\begin{align*}
&(\rho_f\dot{\bfw},\bfv)_{\Omega_f}+(2\mu_f\bfD(\bfw),\bfD(\bfv))_{\Omega_f}+\rho_f(\bfu_1\cdot\nabla \bfu_1-\bfu_2\cdot\nabla \bfu_2,\bfv)_{\Omega_f}-(\phi,\nabla\cdot \bfv)_{\Omega_f}\\
&
+(\rho_s\ddot{\bftheta},\bfxi)_{\Omega_p}+(2\mu_s\bfD(\bftheta),\bfD(\bfxi))_{\Omega_p}+(\lambda_s\nabla\cdot \bftheta,\nabla\cdot\bfxi)_{\Omega_p}-(\alpha \phi,\nabla\cdot \bfxi)_{\Omega_p}\\
&+(s_0\dot{\phi}+\alpha\nabla\cdot \dot{\bftheta},r)_{\Omega_p}+(\bfK\nabla \phi,\nabla r)_{\Omega_p}
\\
&+\langle \phi\bfn_\Gamma,\bfv-\bfxi\rangle_{\Gamma_I}+\Sigma_{l=1}^{d-1}\langle\beta(\bfw-\dot{\bftheta})\cdot \bft_{\Gamma}^l,(\bfv-\bfxi)\cdot \bft_{\Gamma}^l\rangle_{\Gamma_I}+\langle(\dot{\bftheta}-\bfw)\cdot\bfn_{\Gamma},r\rangle_{\Gamma_I}\\
&\hspace{3cm}=0
\end{align*}
for all $\bfv, \bfxi, r$. Let $\bfv=\bfw, \bfxi=\dot{\bftheta}, r=\phi$ where $\bfw(0)=\bfzero, \bftheta(0)=\bfzero, \dot{\bftheta}(0)=\bfzero, \phi(0)=0$.
\begin{align*}
&\dfrac{\rho_f}2\dfrac{d}{dt}\|\bfw\|^2_{\Omega_f}+2\mu_f\|\bfD(\bfw)\|^2_{\Omega_f}
+\dfrac{\rho_s}2\dfrac{d}{dt}\|\dot{\bftheta}\|^2_{\Omega_p}+\mu_s\dfrac{d}{dt}\|\bfD(\bftheta)\|^2_{\Omega_p}\\
&+\dfrac{\lambda_s}2\dfrac{d}{dt}\|\nabla\cdot \bftheta\|^2_{\Omega_p}+\dfrac{s_0}2\dfrac{d}{dt}\|\phi\|^2_{\Omega_p}+\|\bfK^{1/2}\nabla \phi\|^2_{\Omega_p}
+\Sigma_{l=1}^{d-1}\beta\|\bfw-\dot{\bftheta})\cdot \bft_{\Gamma}^l\|^2_{\Gamma_I}\\
%&= -\rho_f(\bfu_1\cdot\nabla \bfu_1-\bfu_2\cdot\nabla \bfu_2,\bfw)_{\Omega_f}
&=-\rho_f(\bfw\cdot\nabla \bfu_1+\bfu_2\cdot\nabla \bfw,\bfw)_{\Omega_f}\\
%&\leq \rho_f\left( \|\bfw\|^2_{L^4(\Omega_f)}\|\nabla \bfu_1\|_{\Omega_f}+\|\bfu_2\|_{L^4(\Omega_f)}\|\nabla \bfw\|_{\Omega_f}\|\bfw\|_{L^4(\Omega_f)}\right)\\
&\leq S_f^2K_f^3\|\bfD(\bfw)\|^2_{\Omega_f}\|\bfD(\bfu_1)\|_{\Omega_f}+S_f^2K_f^3\|\bfD(\bfu_2)\|_{\Omega_f}\|\bfD(\bfw)\|_{\Omega_f}^2
\leq 2\mu_f\|\bfD(\bfw)\|^2_{\Omega_f}.
\end{align*}
Therefore, 
\[
\dfrac{\rho_f}2\dfrac{d}{dt}\|\bfw\|^2_{\Omega_f}
+\dfrac{\rho_s}2\dfrac{d}{dt}\|\dot{\bftheta}\|^2_{\Omega_p}+\mu_s\dfrac{d}{dt}\|\bfD(\bftheta)\|^2_{\Omega_p}\\
+\dfrac{\lambda_s}2\dfrac{d}{dt}\|\nabla\cdot \bftheta\|^2_{\Omega_p}+\dfrac{s_0}2\dfrac{d}{dt}\|\phi\|^2_{\Omega_p}\leq 0
\]
which implies using the initial conditions:
\[
\dfrac{\rho_f}2\|\bfw(t)\|^2_{\Omega_f}
+\dfrac{\rho_s}2\|\dot{\bftheta}(t)\|^2_{\Omega_p}+\mu_s\|\bfD(\bftheta(t))\|^2_{\Omega_p}
+\dfrac{\lambda_s}2\|\nabla\cdot \bftheta\|^2_{\Omega_p}+\dfrac{s_0}2\|\phi\|^2_{\Omega_p}\leq 0
\]
This implies
\[
\bfw=\bfzero, \bftheta=\bfzero, \phi=0.
\]
\end{proof}
\subsection{Existence of a Navier-Stokes pressure.}\label{sec:NSEpressure}
Next we prove that from any solution of {\bf (WF2)} we can recover a solution of {\bf{(WF1)}}. In fact, an inf-sup condition is sufficient to show the existence of a $p_f$ which was eliminated when we restricted the search space $L^2(0,T;\bfX_f)$ to $L^2(0,T;\bfV_f)$.
The following inf-sup condition holds \cite{GR2009, CGR2013}: There exists a constant $\kappa>0$ such that 
\[
\inf_{q\in Q_f}\sup_{\bfv\in \bfX_f}\dfrac{(\nabla\cdot\bfv,q)_{\Omega_f}}{\|\bfv\|_{H^1(\Omega_f)}\|q\|_{L^2(\Omega_f)}}\geq \kappa.
\]
Note that we can replace the supremum above with the supremum over all $(\bfv,\bfxi,r)\in \bfX_f\times \bfX_p\times Q_p$. In fact, the supremum is attained when $\bfxi=\bfzero$ and $r=0$.
\begin{lemma}
If $(\bfu_f, \bfeta, p_p)$ is a solution of problem {\bf (WF2)}, then there exists a unique $p_f\in L^{\infty}(0,T;Q_f)$ such that $(\bfu_f,p_f,\bfeta, p_p)$ is a solution of problem \bf{(WF1)}.
\end{lemma}
\begin{proof}
Let's define a mapping $\calF$ such that for all $(\bfv,\bfxi,r)\in \bfX_f\times \bfX_p\times Q_p$ 
\begin{align*}
&\calF(\bfv, \bfxi, r)=(\rho_f\dot{\bfu}_f,\bfv)_{\Omega_f}+(2\mu_f\bfD(\bfu_f),\bfD(\bfv))_{\Omega_f}+\rho_f(\bfu_f\cdot\nabla \bfu_f,\bfv)_{\Omega_f}
\\&+(\rho_s\ddot{\bfeta},\bfxi)_{\Omega_p}+(2\mu_s\bfD(\bfeta),\bfD(\bfxi))_{\Omega_p}+(\lambda_s\nabla\cdot \bfeta,\nabla\cdot\bfxi)_{\Omega_p}-(\alpha p_p(t),\nabla\cdot \bfxi)_{\Omega_p}\\
&+(s_0\dot{p_p}+\alpha\nabla\cdot \dot{\bfeta},r)_{\Omega_p}+(\bfK\nabla p_p,\nabla r)_{\Omega_p}
\\
&+\langle p_p\bfn_\Gamma,\bfv-\bfxi\rangle_{\Gamma_I}+\Sigma_{l=1}^{d-1}\langle\beta(\bfu_f-\bfetadot)\cdot \bft_{\Gamma}^l,(\bfv-\bfxi)\cdot \bft_{\Gamma}^l\rangle_{\Gamma_I}+\langle(\bfetadot-\bfu_f)\cdot\bfn_{\Gamma},r\rangle_{\Gamma_I}\\
&+\langle P_{in}(t)\bfn_f,\bfv\rangle_{\Gamma_f^{in}}-(\bff_f,\bfv)_{\Omega_f}-(\bff_s,\bfxi)_{\Omega_p}-(f_p,r)_{\Omega_p}, \mbox{ a.e. in }(0,T).
\end{align*}
It is straightforward to see that $\calF$ is linear and continuous on $\bfX_f\times \bfX_p\times Q_p$ for a.e. $t\in (0,T)$. Furthermore, $\calF(\bfv,\bfxi,r)=0$ for any $(\bfv,\bfxi,r)\in \bfV_f\times \bfX_f\times Q_p$ for a.e. $t\in (0,T)$. Therefore, the theory of Babuska-Brezzi imply for a.e. $t\in (0,T)$ that there exists a unique function $p_f\in Q_f$ such that 
\begin{equation}\label{NSEpressure}
(p_f,\nabla\cdot \bfv)_{\Omega_f}=\calF(\bfv, \bfxi, r), \forall (\bfv,\bfxi,r)\in \bfX_f\times \bfX_f\times Q_p, 
\end{equation}
that is, there is a unique $p_f\in L^{\infty}(0,T;Q_f)$ such that $(\bfu_f,p_f,\bfeta, p_p)$ is a solution of problem {\bf (WF1)}. Furthermore, letting $r=0$, $\bfxi=\bfzero$ in \eqref{NSEpressure} and using the inf-sup condition again 
we have 
\begin{align*}
\|p_f\|_{\Omega_f}\leq \dfrac1{\kappa}\Big(\rho_f\|\dot{\bfu}_f\|_{\Omega_f}+2\mu_f\|\bfD(\bfu_f)\|_{\Omega_f}+S_f^2\|\bfu_f\|_{1, \Omega_f}^2+T_1T_3\|p_p\|_{1,\Omega_p}\\+\beta T_1^2\|\bfu_f\|_{1,\Omega_f}+\beta T_1T_5\|\dot{\bfeta}\|_{1,\Omega_p}+T_2\|P_{in}\|_{\Gamma_f^{in}}+\|\bff_f\|_{\Omega_f}\Big),
\end{align*}
a.e. in $(0,T)$ which implies the bound \eqref{pfbound} on $p_f$ in Theorem \ref{mainresult}.
\end{proof}
This concludes the proof of Theorem \ref{mainresult}.
\appendix 
\section{The interface conditions.} \label{justification}
In this section we prove that the interface conditions are meaningful for a solution of \eqref{moment}, \eqref{mass}, \eqref{structure} and \eqref{mass1} 
and therefore justify the derivation of the weak formulation  in the proof of Proposition \ref{equivalence}.

Let us consider a solution of \eqref{moment}, \eqref{mass}, \eqref{structure} and \eqref{mass1} such that $\bfu_f\in L^2(0,T;\bfX_f)$, $p_f\in L^1(0,T;Q_f)$, $\bfeta\in H^1(0,T;\bfX_p)$ and  $p_p \in L^2(0,T;Q_p)$ with $\bfudot_f\in L^1(0,T;\bfL^{\frac32}(\Omega_f))$. 
\subsubsection*{Interface conditions \eqref{inter3} and \eqref{inter4}.}
Since $\bfu_f\in L^2(0,T;\bfX_f)$ and $\bfH^1(\Omega_f)$ is  embedded continuously in $\bfL^6(\Omega_f)$, for $d=2,3$, $\bfu_f\in L^2(0,T;\bfL^6(\Omega_f))$.
Then H\"{o}lder's inequality implies 
\begin{align*}
\|\bfu_f\cdot\nabla\bfu_f\|_{L^1(0,T;\bfL^{3/2}(\Omega_f)}
%&=\int_0^T\|\bfu_f\cdot\nabla\bfu_f\|_{\bfL^{3/2}(\Omega_fw)}dt\leq \int_0^T\left(\int_{\Omega_f}|\bfu_f\cdot\nabla\bfu_f|^{3/2}dV\right)^{2/3}dt\\
%&\leq  \int_0^T\left(\int_{\Omega_f}|\bfu_f|^{3/2}|\nabla\bfu_f|^{3/2}dV\right)^{2/3}dt\\
%&\leq \int_0^T \left(\Big(\int_{\Omega_f}|\bfu_f|^{6}dV\Big)^{1/4}\Big(\int_{\Omega_f}|\nabla\bfu_f|^{2}dV\Big)^{3/4}\right)^{2/3}dt\\
%&\leq \int_0^T \Big(\int_{\Omega_f}|\bfu_f|^{6}dV\Big)^{1/6}\Big(\int_{\Omega_f}|\nabla\bfu_f|^{2}dV\Big)^{1/2}dt\\
%&\leq \int_0^T\|\bfu_f\|_{\bfL^6(\Omega_f)}\|\nabla\bfu_f\|_{\bfL^2(\Omega_f)}dt\\
&\leq\|\bfu_f\|_{L^2(0,T;\bfL^6(\Omega_f))}\|\bfu_f\|_{L^2(0,T;\bfX_f)}.
\end{align*}
Therefore, $\bfu_f\cdot\nabla\bfu_f\in L^1(0,T;\bfL^{3/2}(\Omega_f))$. 
Then from \eqref{moment}
\[
\nabla\cdot \bfsigma_f=\bff_f-\rho_f\bfudot_f-\bfu_f\cdot\nabla \bfu_f\in L^1(0,T;\bfL^{3/2}(\Omega_f)).
\]
Also 
\[
\bfsigma_f = 2\mu_fD(\bfu_f)-p_fI \in L^1(0,T;\bfL^2(\Omega_f)).
\]
Hence, each row of $\bfsigma_f$ belongs to $L^1(0,T;\bfH^{3/2}(\dive;\Omega_f))$. 
Since $\bfH^1(\Omega_f)$ is dense in $\bfH^{3/2}({\dive};\Omega_f)$, the following Green's formula hold:
\[
\forall \bfv\in \bfH^{3/2}({\dive};\Omega_f), \forall \phi\in H^1(\Omega_f), \langle \bfv\cdot \bfn_{\partial \Omega_f}, \phi \rangle_{\partial \Omega_f}=(\nabla \cdot \bfv, \phi)_{\Omega_f}+(\bfv, \nabla\phi)_{\Omega_f}.
\]
This allows $\bfsigma_f\bfn_{\Gamma}$ to be defined in a weak sense on $\Gamma_I$. See also \cite[p.541]{CGR2013}.
 
Next  since $\bfu\in L^2(0,T;\bfX_f)$, $\dot{\bfeta}\in L^2(0,T;\bfX_p)$ and $p_p\in L^2(0,T;Q_p)$ we have $(-p_p\bfn_{\Gamma}-\sum_{l=1}^{d-1}(\beta(\bfu_f-\bfetadot)\cdot \bft_{\Gamma}^l)\bft_{\Gamma}^l)\Big|_{\Gamma_I}\in L^2(0,T;\bfL^4(\Gamma_I))$
where $\bft_{\Gamma}^1, \bft_{\Gamma}^2$ are the unit tangential vectors on $\Gamma_I$. With these the interface conditions \eqref{inter3} and \eqref{inter4} make sense.

For the rest of the interface conditions we use a global space-time argument in $\mathbb{R}^{d+1}$, $d=2, 3$. We define new variables in $\mathbb{R}^{d+1}$ by $\tilde{\bfx}=(t, x_1, \hdots, x_d)$, $d=2, 3$ and also define the cylindrical region $\tilde{\Omega}_p=[0,T]\times \Omega_p$. 
\subsubsection*{Interface condition \eqref{inter1}.} 
Defining ${\bf \Theta}=(-(s_0p_p+\alpha\nabla\cdot \bfeta), \bfK\nabla p_p)$,  the equation \eqref{mass1} can be written as
\[-\nabla_{\tilde{\bfx}}\cdot{\bf \Theta}=f_p\]
where $\nabla_{\tilde{\bfx}}=(\dfrac{d}{dt},\nabla\cdot)^T$.

Since $f_p\in L^2(\tilde{\Omega}_p)$ and ${\bf \Theta}\in \bfL^2(\tilde{\Omega}_p)$, we have ${\bf \Theta}\in \bfH(\dive;\tilde{\Omega}_p)$. So ${\bf \Theta}\cdot \bfn_{\tilde{\Omega}_p}|_{\partial \tilde{\Omega}_p}$ is well-defined in $H^{-1/2}(\partial \tilde{\Omega}_p)$. Therefore the following Green's formula holds \cite{GR1986}:
\[
\forall \tilde{\bfphi}\in \bfH^1(\tilde{\Omega}_p), \quad (\nabla_{\tilde{\bfx}}\cdot {\bf \Theta}, \tilde{\bfphi})_{\tilde{\Omega}_p}=-({\bf \Theta},\nabla_{\tilde{\bfx}} \tilde{\bfphi})_{\tilde{\Omega}_p}+\langle ({\bf \Theta}\cdot \bfn_{\tilde{\Omega}_p}, \tilde{\bfphi}\rangle_{\partial \tilde{\Omega}_p}
\]
where $\bfn_{\tilde{\Omega}_p}$ outward unit normal to $\tilde{\Omega}_p$. 

This allows the well-definition of $\bfK\nabla p_p\cdot \bfn_{\partial \Omega_p}$ in the weak sense 
 on $\Gamma_I$.
\subsubsection*{Interface condition \eqref{inter2}.}
From the above discussion it is enough to check whether $\bfsigma_p\bfn_{\Gamma}$ is meaningful and we use a similar space-time argument. Defining ${\bf \Sigma}=(-\rho_s \dot{\bfeta}, \bfsigma_p)$, the equation \eqref{structure} can be written as
\[
-\nabla_{\tilde{\bfx}}\cdot{\bf \Sigma}=\bff_s.
\]
Since $\bff_s\in L^2(\tilde{\Omega})$ and $\bfSigma\in \bfL^2(\tilde{\Omega}_p)$ we have ${\bf \Sigma}\in \bfH(\dive;\tilde{\Omega}_p)$. Therefore the following Green's formula holds \cite{GR1986}:
\[
\forall \tilde{\bfphi}\in \bfH^1(\tilde{\Omega}_p), \quad (\nabla_{\tilde{\bfx}}\cdot {\bf \Sigma}, \tilde{\bfphi})_{\tilde{\Omega}_p}=-({\bf \Sigma},\nabla_{\tilde{\bfx}} \tilde{\bfphi})_{\tilde{\Omega}_p}+\langle {\bf \Sigma}\cdot \bfn_{\tilde{\Omega}_p}, \tilde{\bfphi}\rangle_{\partial \tilde{\Omega}_p}.
\]
This allows the well-definition of $\bfsigma_p\bfn_{\Gamma}$ in the weak sense %$(\tilde{\bfu}\cdots \bfn_{\tilde{\Omega}_p}$ in $H^{-1/2}(\partial \tilde{\Omega}_p)$.
 on $\Gamma_I$.

\section*{References}
\bibliographystyle{elsarticle-harv} 
\bibliography{mybibfile}



\end{document}